\documentclass[psamsfonts]{amsart}

\usepackage[a4paper,bindingoffset=0.0in,left=.975in,right=.975in,top=1.50in,bottom=1.50in,footskip=.35in]{geometry}

\usepackage{amssymb,amsfonts}
\usepackage[all,arc]{xy}
\usepackage{enumerate}
\usepackage{mathrsfs}

\usepackage{setspace}

\newtheorem{thm}{Theorem}[section]
\newtheorem{cor}[thm]{Corollary}
\newtheorem{prop}[thm]{Proposition}
\newtheorem{lem}[thm]{Lemma}

\theoremstyle{definition}
\newtheorem{defn}[thm]{Definition}

\theoremstyle{remark}
\newtheorem{rem}[thm]{Remark}

\newcommand{\A}{\Bbb{A}}

\newcommand{\R}{\Bbb{R}}

\newcommand{\C}{\Bbb{C}}

\newcommand{\Q}{\Bbb{Q}}

\newcommand{\OO}{\mathcal{O}}

\newcommand{\CC}{\mathcal{C}}

\newcommand{\WW}{\mathcal{W}}

\newcommand{\FF}{\mathcal{F}}

\newcommand{\HH}{\mathcal{H}}

\newcommand{\F}{\Bbb{F}}

\newcommand{\G}{\Bbb{G}}

\newcommand{\Z}{\Bbb{Z}}

\newcommand{\U}{\Bbb{U}}

\newcommand{\B}{\Bbb{B}}

\newcommand{\T}{\Bbb{T}}

\newcommand{\p}{\frak{p}}

\newcommand{\m}{\frak{m}}

\newcommand{\M}{\mathcal{M}}

\newcommand{\WD}{\text{WD}}

\newcommand{\Mod}{\text{Mod}}

\newcommand{\adm}{\text{adm}}

\newcommand{\an}{\text{an}}

\newcommand{\Ind}{\text{Ind}}

\newcommand{\sm}{\text{sm}}

\newcommand{\alg}{\text{alg}}

\newcommand{\Sp}{\text{Sp}}

\newcommand{\cl}{\text{cl}}

\newcommand{\Spf}{\text{Spf}}

\newcommand{\rig}{\text{rig}}

\newcommand{\rec}{\text{rec}}

\newcommand{\Ban}{\text{Ban}}

\newcommand{\diag}{\text{diag}}

\newcommand{\Frob}{\text{Frob}}

\newcommand{\Hom}{\text{Hom}}

\newcommand{\Ext}{\text{Ext}}

\newcommand{\End}{\text{End}}

\newcommand{\tr}{\text{tr}}

\newcommand{\Gal}{\text{Gal}}

\newcommand{\GL}{\text{GL}}

\newcommand{\BC}{\text{BC}}

\newcommand{\Spec}{\text{Spec}}

\newcommand{\y}{\hspace{6pt}}

\newenvironment{psmallmatrix}
  {\left(\begin{smallmatrix}}
  {\end{smallmatrix}\right)}
  
\makeatletter
\let\c@equation\c@thm
\makeatother
\numberwithin{equation}{section}

\title{Local Langlands correspondence in rigid families}

\author{Christian Johansson \y \y James Newton \y \y Claus Sorensen}

\date{}

\begin{document}

\begin{abstract}
We show that local-global compatibility (at split primes) away from $p$ holds at {\it{all}} points of the $p$-adic eigenvariety of a definite $n$-variable unitary group.
The novelty is we allow non-classical points, possibly non-\'{e}tale over weight space. More precisely we interpolate the local Langlands correspondence for $\GL_n$ across the eigenvariety by considering the fibers of its defining coherent sheaf. We employ techniques of Scholze from his new approach to the local Langlands conjecture.
\end{abstract}

\maketitle

\tableofcontents

%-------------------------

\onehalfspacing

\section{Introduction}

The proof of (strong) local-global compatibility in the $p$-adic Langlands program for $\GL(2)_{/\Q}$, cf. \cite{Em11b}, led Emerton and Helm to speculate on whether the local Langlands correspondence for $\GL(n)$ can be interpolated in Zariski families: Say $L/\Q_{\ell}$ is a finite extension, with absolute Galois group $\Gal_L$, and $R$ is a reduced complete local Noetherian ring with residual characteristic $p \neq \ell$. Assume $R$ is $p$-torsion free with finite residue field $R/\m_R$.
If $\rho: \Gal_L \rightarrow \GL_n(R)$ is a continuous representation it is shown in \cite{EH14} (cf. their Theorem 1.2.1) that there is at {\it{most}} one admissible $\GL_n(L)$-representation $\tilde{\pi}(\rho)$ over $R$ satisfying a short list of natural desiderata:
\begin{itemize}
\item[(a)] $\tilde{\pi}(\rho)$ is $R$-torsion free (meaning $\text{Ann}_R (\tilde{\pi}(\rho))$ consists of zero-divisors);
\item[(b)] If $\p \subset R$ is a minimal prime with residue field $\kappa(\p)=\text{Frac}(R/\p)$ (which is a possibly non-algebraic extension of $\Q_p$ as the example $\Z_p[[X]]$ shows) then
$$
\tilde{\pi}(\rho \otimes_R \kappa(\p)) \overset{\sim}{\longrightarrow} \tilde{\pi}(\rho)\otimes_R \kappa(\p),
$$
where the source is the representation of $\GL_n(L)$ associated with $\rho \otimes_R \kappa(\p)$ under the local Langlands correspondence (suitably modified and normalized);
cf. \cite[Sect. 4]{EH14}.
\item[(c)] The reduction $\tilde{\pi}(\rho)\otimes_R R/\m_R$ has an absolutely irreducible and generic cosocle (i.e., the largest semisimple quotient), and no other generic constituents. 
\end{itemize}
They conjecture (\cite[Conj. 1.4.1]{EH14}) that such a $\tilde{\pi}(\rho)$ always exists, and this is now known in many cases, due to ongoing work of Helm and Moss; in particular when $p$ is banal (meaning that 
$1,q,\ldots,q^n$ are distinct modulo $p$, where $q$ is the size of the residue field of $L$), see section 8 of \cite{Hel12}. 

In this paper we take a different point of view, and interpolate the local Langlands correspondence in {\it{rigid}} families -- across eigenvarieties of definite unitary groups; in much the same spirit as \cite{Pau11} for the eigencurve, except that our approach is more representation-theoretic than geometric and based on the construction of eigenvarieties in \cite{Em06}
and \cite{BHS16}. In the next few paragraphs we introduce notation in order to state our main result (Theorem \ref{main} below). 

Let $p>2$ be a prime, and fix an unramified CM extension $F/F^+$ which is split at all places $v$ of $F^+$ above $p$. Suppose $U_{/F^+}$ is a unitary group in $n$ variables which is quasi-split at all finite places and compact at infinity (see \ref{unit} for more details). Throughout $\Sigma$ is a finite set of finite places of $F^+$ containing $\Sigma_p=\{v:v|p\}$, and we let $\Sigma_0=\Sigma\backslash \Sigma_p$. We assume all places $v \in \Sigma$ split in $F$ and we choose a divisor $\tilde{v}|v$ once and for all, which we use to make the identification $U(F_v^+)\overset{\sim}{\longrightarrow} \GL_n(F_{\tilde{v}})$. We do {\it{not}} assume the places in $\Sigma$ are banal (as in \cite{CHT08} for instance). 
We consider tame levels of the form $K^p=K_{\Sigma_0}K^{\Sigma}$  where $K^{\Sigma}=\prod_{v \notin \Sigma}K_v$ is a product of hyperspecial maximal compact subgroups, and $K_{\Sigma_0}=\prod_{v \in \Sigma_0}K_v$. 

Our coefficient field is a sufficiently large finite extension $E/\Q_p$ with integers $\OO$ and residue field $k=k_E$, and we start off with an absolutely irreducible\footnote{This is mostly for convenience. The automorphic $\OO$-lifts of $\bar{r}$ then arise from {\it{cusp}} forms on $\GL_n(\A_F)$, cf. Lem. \ref{unique}. } Galois representation 
$\bar{r}: \Gal_F \rightarrow \GL_n(k)$ which is automorphic of tame level $K^p$. We let $\m=\m_{\bar{r}}$ be the associated maximal ideal, viewed in various Hecke algebras (see sections \ref{hecke} and \ref{gal} for more details). In \ref{def} and \ref{defspace} we introduce the universal deformation ring $R_{\bar{r}}$ and the deformation space $X_{\bar{r}}=\Spf(R_{\bar{r}})^{\rig}$. Each point $x \in X_{\bar{r}}$ carries a Galois representation $r_x$, which is a deformation of $\bar{r}$, and we let $\p_x \subset R_{\bar{r}}$ be the associated prime ideal. The Banach representation of $p$-adic automorphic forms $\hat{S}(K^p,E)_{\m}$ inherits a natural $R_{\bar{r}}$-module structure, and we consider its $\p_x$-torsion $\hat{S}(K^p,E)_{\m}[\p_x]$ and its dense subspace of locally analytic vectors $\hat{S}(K^p,E)_{\m}[\p_x]^{\an}$, cf. section \ref{auto}.

The eigenvariety $Y(K^p,\bar{r})\subset X_{\bar{r}}\times \hat{T}$ is defined as the support of a certain coherent sheaf $\M$ on $X_{\bar{r}}\times \hat{T}$. Here $\hat{T}$ denotes the character space of the $p$-adic torus $T\subset U(F^+\otimes \Q_p)$ isomorphic to $\prod_{v|p}T_{\GL(n)}(F_{\tilde{v}})$, see \ref{charspace} below. We have $\hat{T}\simeq \WW \times (\G_m^{\rig})^{n|\Sigma_p|}$ where $\WW$ is weight space (parametrizing continuous characters of the maximal compact subgroup of $T$) which is a disjoint union of finitely many open unit balls of dimension $n[F^+:\Q]$. By definition a point $y=(x,\delta) \in X_{\bar{r}}\times \hat{T}$ belongs to the eigenvariety $Y(K^p,\bar{r})$ if and only if the fiber $\M_y$
is nonzero. If $y$ is $E$-rational the $E$-linear dual of $\M_y$ can be described as 
$$
\M_y'\simeq J_B^{\delta}(\hat{S}(K^p,E)_{\m}[\p_x]^{\an})
$$
where $J_B$ denotes Emerton's locally analytic variant of the Jacquet functor (\cite{Em06b}) and $J_B^{\delta}$ means the $\delta$-eigenspace. Morally our main result states that 
$\varinjlim_{K_{\Sigma_0}} \M_y'$ interpolates the local Langlands correspondence for $\GL_n$ across the eigenvariety. Here is the precise formulation. 

\begin{thm}\label{main}
Let $y=(x,\delta)\in Y(K^p,\bar{r})(E)$ be a point such that $r_x$ is {\bf{strongly generic}} at every $v \in \Sigma_0$ (cf. Def. \ref{gen} in the main text). Then there is an $m_y \in \Z_{>0}$ such that up to semi-simplification
$$
\varinjlim_{K_{\Sigma_0}} J_B^{\delta}(\hat{S}(K^p,E)_{\m}[\p_x]^{\an})\overset{ss}{\simeq} \big(\otimes_{v\in \Sigma_0}\pi_{x,v}\big)^{\oplus m_y}
$$
as admissible representations of $U(F_{\Sigma_0}^+)\overset{\sim}{\longrightarrow} \prod_{v\in \Sigma_0} \GL_n(F_{\tilde{v}})$, where $\pi_{x,v}$ is the irreducible smooth representation of $U(F_v^+) \overset{\sim}{\longrightarrow} \GL_n(F_{\tilde{v}})$ associated with $r_x|_{\Gal_{F_{\tilde{v}}}}$ via the local Langlands correspondence: 
$$
\WD(r_x|_{\Gal_{F_{\tilde{v}}}})^{F-ss}\simeq \rec(\BC_{\tilde{v}|v}(\pi_{x,v}) \otimes |\det|^{(1-n)/2})
$$
with $\rec(\cdot)$ normalized as in \cite{HT01}.
\end{thm}

The notation $\BC_{\tilde{v}|v}(\pi_{x,v})$ used in the theorem signifies local base change, which simply amounts to viewing $\pi_{x,v}$ as a representation of $\GL_n(F_{\tilde{v}})$ via its identification with $U(F_v^+)$. 

%By $y$ being {\it{\'{e}tale}} we mean that there is a neighborhood $\Omega$ such that the weight map $\omega: Y(K^p,\bar{r}) \rightarrow \WW$ restricts to an isomorphism $\Omega \overset{\sim}{\longrightarrow} \omega(\Omega)$. This may be a bit ad hoc.

The hypothesis that $r_x$ is strongly generic guarantees that $\pi_{x,v}$ is fully induced from a supercuspidal. Our result is actually slightly more general (cf. Lemma \ref{sq}) in that it deals with arbitrary (non-generic) points as well. In that generality one can only say that any irreducible subquotient of 
$\varinjlim_{K_{\Sigma_0}} \M_y'$  has the same supercuspidal support as $\otimes_{v\in \Sigma_0}\pi_{x,v}$. In particular $\varinjlim_{K_{\Sigma_0}} \M_y'$ lies in the Bernstein component $\mathcal{R}^{\frak{s}}(U(F_{\Sigma_0}^+))$ for the inertial class $\frak{s}$ determined by $y$ (cf. section \ref{generic}). We have no information about the monodromy operator. Our methods are based on $p$-adic interpolation of traces and therefore seem inherently inapplicable to deal with non-trivial monodromy. In the case where $K_{\Sigma_0}$ is a product of Iwahori subgroups we can be a little more precise however (see Theorem \ref{iwcase} at the very end); using a genericity criterion of Barbasch-Moy, generalized by Chan-Savin, we show for {\it{any}} point $y$ that 
$\otimes_{v\in \Sigma_0}\pi_{x,v}^{\text{gen}}$ is the only generic constituent of $\varinjlim_{K_{\Sigma_0}} \M_y'$ -- and it {\it{does}} appear -- where $\pi_{x,v}^{\text{gen}}$ denotes the generic representation with the same supercuspidal support as $\pi_{x,v}$. At the other extreme, when $y=(x,\delta)$ is a point for which $\pi_{x,v}$ is supercuspidal for all $v \in \Sigma_0$ we can remove the '$ss$' in Theorem \ref{main} since there are no self-extensions with central character that of $\pi_{x,v}$ (cf. Remark \ref{ext}) by the projectivity and/or injectivity of $\pi_{x,v}$ in this category.  

We expect that the length $m_y$ of $\varinjlim_{K_{\Sigma_0}} \M_y'$ as a $U(F_{\Sigma_0}^+)$-representation can be $>1$ at certain singular points. 
If $y$ is a classical point of non-critical slope (automatically \'{e}tale by \cite[Thm. 4.10]{Che11}) $m_y=1$, cf. Proposition \ref{ncr} below. Under certain mild non-degeneracy assumptions, $m_y$ should be closely related to $\dim_E J_B^{\delta}(\Pi(\varrho_x)^{\an})$; which is finite by \cite[Cor. 0.15]{Em06c}. Here 
$\varrho_x:=\{r_x|_{\Gal_{F_{\tilde{v}}}}\}_{v\in \Sigma_p}$ and
$\Pi(\varrho_x):=\hat{\otimes}_{v|p} \Pi(r_x|_{\Gal_{F_{\tilde{v}}}})$, where $\Pi(\cdot)$ is the $p$-adic local Langlands correspondence for $\GL_n(F_{\tilde{v}})$ -- as defined in
\cite{CEG$^+$16} say, to fix ideas\footnote{At least for the choice of $R_{\infty}\rightarrow \OO$ in \cite{CEG$^+$16} compatible with $x:R_{\bar{r}}\rightarrow \OO$ via the projection $R_{\infty} \twoheadrightarrow R_{\bar{r}}$.}. This expectation is based on the strong local-global compatibility results of \cite{Em11b} and \cite{CS17}, which also seem to suggest that 
$\varinjlim_{K_{\Sigma_0}} \M_y'$ should in fact be semisimple -- for generic points (otherwise the 'generic' local Langlands correspondence gives a reducible indecomposable representation). We are not sure if this is an artifact of the $n=2$ case, or this is supposed to be true more generally. It is certainly not true for trivial reasons since $\pi_{x,v}$ does admit non-trivial self-extensions. For example, by \cite[Cor. 2]{Orl05} we have $\dim\Ext^i_{\GL_n}(\text{St},\text{St})=\binom{n}{i}$. 
Even when $\pi_{x,v}$ is parabolically induced from a supercuspidal it does happen that $\Ext_{\GL_n(F_{\tilde{v}})}^1(\pi_{x,v},\pi_{x,v})\neq 0$ (cf. Remark \ref{ext}.)

Theorem \ref{main} was motivated in part by the question of local-global compatibility for the Breuil-Herzig construction $\Pi(\rho)^{\text{ord}}$, cf. \cite[Conj. 4.2.5]{BH15}. The latter 
is defined for upper triangular $p$-adic representations $\rho$ of $\Gal_{\Q_p}$, and is supposed to model the largest subrepresentation of the 'true' $p$-adic local Langlands correspondence built from unitary continuous principal series representations. We approach this problem starting from the inclusion (for unitary $\delta$)
\begin{equation}\label{jacord}
J_B^{\delta}(\hat{S}(K^p,E)_{\m}[\p_x]^{\an}) \hookrightarrow \text{Ord}_B^{\delta}(\hat{S}(K^p,\OO)_{\m}[\p_x])[1/p]^{\an},
\end{equation}
as shown in \cite[Thm. 6.2]{Sor15}. Here $\text{Ord}_B$ is Emerton's functor of ordinary parts (\cite{Em10}), which is right adjoint to parabolic induction $\text{Ind}_{\bar{B}}$.
If $y=(x,\delta)$ lies on $Y(K^p,\bar{r})$ the source of (\ref{jacord}) is nonzero, and we deduce the existence of a nonzero (norm-decreasing) equivariant map 
$\text{Ind}_{\bar{B}}(\delta) \rightarrow \hat{S}(K^p,E)_{\m}[\p_x]$. If one could show that certain Weyl-conjugates $y_w=(x,w\delta)$ all lie on $Y(K^p,\bar{r})$ one would infer that
there is a non-trivial map $\text{soc}_{\GL_n(\Q_p)} \Pi(\rho)^{\text{ord}}\rightarrow \hat{S}(K^p,E)_{\m}[\p_x]$ which one could hope to promote to a map $\Pi(\rho)^{\text{ord}}\rightarrow \hat{S}(K^p,E)_{\m}[\p_x]$ using \cite[Cor. 4.3.11]{BH15}. Here we take $\rho=r_x|_{\Gal_{F_{\tilde{v}}}}$ (up to a twist which we ignore here) for some $v|p$ such that $F_{\tilde{v}}=\Q_p$, and $x$ is a point where $r_x|_{\Gal_{F_{\tilde{v}}}}$ is upper triangular with $\delta_{\tilde{v}}$ on the diagonal. With a little more work, it is conceivable that Theorem \ref{main} would give {\it{strong}} local-global compatibility -- in the sense that there is an embedding
$$
\otimes_{v\in \Sigma_0}\pi_{x,v} \hookrightarrow \varinjlim_{K_{\Sigma_0}} \Hom_{\GL_n(\Q_p)}(\Pi(\rho)^{\text{ord}}, \hat{S}(K^p,E)_{\m}[\p_x]).
$$
Some of us hope to return to these questions on another occasion. 

Let us succinctly draw a comparison with \cite{BC09} which contains results of the same nature (see their section 7.4, p. 179): They consider an eigenvariety $X$ of 'idempotent type' and a finite set of places $S$ away from $p$. On $X$ they construct a sheaf $\Pi_S$ of admissible $G(\A_S)$-representations and study how the fibers $\Pi_{S,x}$ vary with $x \in X$. 
Each point $x$ has an associated Hecke eigensystem $\psi_x: \mathcal{H} \rightarrow \kappa(x)$ and one considers a certain generalized eigenspace $\mathcal{S}^{\psi_x}$ of $p$-adic automorphic forms; $\Pi_S^{\psi_x}$ is then the $G(\A_S)$-representation over $\OO_{X,x}/\m_{\omega(x)} \OO_{X,x}$ generated by $\mathcal{S}^{\psi_x}$. It is of finite length and has the interpolative property that there exists a surjection $\Pi_{S,x}/\m_{\omega(x)} \Pi_{S,x} \twoheadrightarrow \Pi_S^{\psi_x}$ for all points $x$. For classical points $x$ there is an inclusion $\Pi_{x}\hookrightarrow \Pi_S^{\psi_x}$, and this is an isomorphism if $x$ is (numerically) non-critical. (We will not recall all the notation used in this paragraph but refer the reader to \cite{BC09}.) A rough 'dictionary' between this paper and \cite{BC09} is  
$$
G(\A_S)\leftrightsquigarrow U(F_{\Sigma_0}^+) \y \y \y \y \y \Pi_{S,x} \leftrightsquigarrow \varinjlim_{K_{\Sigma_0}} \M_y' \y \y \y \y \y \Pi_S^{\psi_x} \leftrightsquigarrow \otimes _{v\in \Sigma_0}\pi_{x,v}.
$$
To summarize, the approach in \cite{BC09} is more geometric and via automorphic {\it{forms}} (akin to \cite{Pau11}) whereas our approach is to interpolate automorphic {\it{representations}} instead of Hecke eigensystems, adopting the definition of eigenvarieties in \cite{BHS16}.

We briefly outline the overall strategy behind the proof of Theorem \ref{main}: For classical points $y=(x,\delta)$ (i.e., those corresponding to automorphic representations) local-global compatibility away from $p$ essentially gives an inclusion $\otimes_{v\in \Sigma_0}\pi_{x,v}\hookrightarrow \varinjlim_{K_{\Sigma_0}} \M_y'$ which is an isomorphism if $\delta$ moreover is of 
non-critical slope. We reinterpret this using ideas from Scholze's proof of the local Langlands correspondence (\cite{Sc13}): He works with certain elements $f_{\tau}$ in the Bernstein center of $\GL_n(F_w)$, associated with $\tau \in W_{F_w}$, which act on an irreducible smooth representation $\Pi$ via scaling by $\tr(\tau|\rec(\Pi))$; here and throughout this paragraph we ignore a twist by $|\det|^{(1-n)/2}$ for simplicity. For each tuple $\tau=(\tau_{\tilde{v}})\in \prod_{v\in \Sigma_0} W_{F_{\tilde{v}}}$ we thus have an element 
$f_{\tau}:=\otimes_{v\in \Sigma_0} f_{\tau_{\tilde{v}}}$ of the Bernstein center of $U(F_{\Sigma_0}^+)\overset{\sim}{\longrightarrow} \prod_{v\in \Sigma_0} \GL_n(F_{\tilde{v}})$ which we know how to evaluate on all irreducible smooth representations. In particular $f_{\tau}$ acts on $\varinjlim_{K_{\Sigma_0}} \M_y'$ via scaling by 
$\prod_{v\in \Sigma_0}\tr \big(\tau_{\tilde{v}}|\rec(\BC_{\tilde{v}|v}(\pi_{x,v}))\big )$ -- still assuming $y$ is classical and non-critical. Those points are Zariski dense in $Y(K^p,\bar{r})$, and using this we interpolate this key scaling property to {\it{all}} points $y$ as follows. By mimicking the standard proof of Grothendieck's monodromy theorem one can interpolate $\WD(r_x|_{\Gal_{F_{\tilde{v}}}})$ in families. Namely, for each $\Sp(A) \subset X_{\bar{r}}$ we construct a Weil-Deligne representation $\WD_{\bar{r},\tilde{v}}$ over $A$ which specializes to $\WD(r_x|_{\Gal_{F_{\tilde{v}}}})$ for all $x \in \Sp(A)$. Around the point $y$ we find a neighborhood $\Omega\subset \Sp(A) \times \hat{T}$ and use the 
weight morphism $\omega: Y(K^p,\bar{r})\rightarrow \WW$, or rather its restriction $\omega|_{\Omega}$, to view $\Gamma(\Omega,\M)$ as a finite type projective module over $\OO_{\WW}(\omega(\Omega))$, which allows us to show that $f_{\tau}$ acts on $\varinjlim_{K_{\Sigma_0}} \Gamma(\Omega,\M)$
via scaling by $\prod_{v\in \Sigma_0} \tr(\tau_{\tilde{v}}|\WD_{\bar{r},\tilde{v}})$. This is the most technical part of our argument; in fact we glue and get the scaling property on the
sheaf $\M$ itself. By specialization at $y$ we deduce that $f_{\tau}$ acts on $\varinjlim_{K_{\Sigma_0}} \M_y'$ via scaling by $\prod_{v\in \Sigma_0}\tr \big(\tau_{\tilde{v}}|\rec(\BC_{\tilde{v}|v}(\pi_{x,v}))\big )$ as desired. This result tells us that every irreducible constituent $\otimes_{v\in \Sigma_0}\pi_v$ of $\varinjlim_{K_{\Sigma_0}} \M_y'$ has the same supercuspidal support as $\otimes_{v\in \Sigma_0}\pi_{x,v}$, and therefore is isomorphic to it if $x$ is a strongly generic point. We also infer that $\varinjlim_{K_{\Sigma_0}} \M_y'$ has finite length since $\dim \M_y' <\infty$ and the constituents 
$\otimes_{v\in \Sigma_0}\pi_v$ have conductors bounded by the conductors of $\WD(r_x|_{\Gal_{F_{\tilde{v}}}})$.

We finish with a few remarks on the structure of the paper. In our first (rather lengthy) section \ref{notation} we introduce in detail the notation and assumptions in force throughout; the
unitary groups $U_{/F^+}$, automorphic forms $\hat{S}(K^p,E)$, Hecke algebras, Galois representations and their deformations. Section \ref{ev} then defines the eigenvarieties 
$Y(K^p,\bar{r})$ and the sheaves $\M_{K^p}$, essentially following \cite{BHS16} and \cite{Em06}. In section \ref{noncrit} we recall the notion of a non-critical classical point, and prove Theorem \ref{main} for those. Section \ref{weildel} interpolates the Weil-Deligne representations across reduced $\Sp(A)\subset X_{\bar{r}}$ by suitably adapting Grothendieck's argument. We recall Scholze's characterization of the local Langlands correspondence in section \ref{llc}, and introduce the functions $f_{\tau}$ in the Bernstein center. The goal of section \ref{tr} is to show Proposition \ref{fam} on the action of $f_{\tau}$ on $\varinjlim_{K_{\Sigma_0}}\Gamma(\Omega,\M_{K^p})$ where $\Omega$ is a neighborhood of $y$ as above. Finally in section \ref{pf} we put the pieces together; we introduce the notion of a strongly generic point, and prove our main results. The last section \ref{iwahori} focuses on the case where $K_{\Sigma_0}$ is a product of Iwahori subgroups; we recall and use the genericity criterion of Chan-Savin to show the occurrence of $\otimes_{v\in \Sigma_0}\pi_{x,v}^{\text{gen}}$.

\medskip

\noindent {\it{Acknowledgements}}. We would like to thank Hauseux and D. Prasad for their assistance with $\Ext^i$, as well as Kaletha, Minguez, and Shin for helpful conversations and correspondence on the details of \cite{KMSW}. Many thanks are due to Savin for his thorough answers to our questions on types and genericity. 

\section{Notation and terminology}\label{notation}

We denote the absolute Galois group $\Gal(F^{\text{sep}}/F)$ of a field $F$ by $\Gal_F$.

\subsection{Unitary groups}\label{unit}

Our setup will be identical to that of \cite{BHS16} although we will adopt a slightly different notation, which we will introduce below. 

We fix a CM field $F$ with maximal totally real subfield $F^+$ and $\Gal(F/F^+)=\{1,c\}$. We assume the extension $F/F^+$ is unramified at all finite places, and split at all places $v|p$ of $F^+$ above a fixed prime $p$.

Let $n$ be a positive integer. If $n$ is even assume that $\frac{n}{2}[F^+:\Q]\equiv 0$ mod $2$. By \cite[3.5]{CHT08}
this guarantees the existence of a unitary group $U_{/F^+}$ in $n$ variables such that
\begin{itemize}
\item $U \times_{F^+} F \overset{\sim}{\longrightarrow} \GL_n$,
\item $U$ is quasi-split over $F_v^+$ (hence unramified) for all\footnote{Convenient in Lem. \ref{unique} when considering local base change from $U(F_v^+)$ to $\GL_n(F_{\tilde{v}})$ -- for unramified representations.} finite places $v$,
\item $U(F^+\otimes_{\Q}\R)$ is compact.
\end{itemize}
We let $G=\text{Res}_{F^+/\Q}U$ be its restriction of scalars. 

If $v$ splits in $F$ the choice of a divisor $w|v$ determines an isomorphism
$i_w: U(F_v^+) \overset{\sim}{\longrightarrow} \GL_n(F_w)$ well-defined up to conjugacy. Throughout we fix a finite set $\Sigma$ of finite places of $F^+$
such that every $v \in \Sigma$ splits in $F$, and $\Sigma$ contains $\Sigma_p=\{v:v|p\}$. We let $\Sigma_0=\Sigma \backslash \Sigma_p$. We emphasize that unlike \cite{CHT08}
we do {\it{not}} assume the places in $\Sigma_0$ are banal.

For each $v \in \Sigma$ we choose 
a divisor $\tilde{v}|v$ once and for all and let $\tilde{\Sigma}=\{\tilde{v}:v \in \Sigma\}$. We also choose an embedding $\Gal_{F_{\tilde{v}}}\hookrightarrow \Gal_F$ for each such $v$.
Moreover, we choose isomorphisms $i_{\tilde{v}}$ which we will tacitly use to identify 
$U(F_v^+)$ with $\GL_n(F_{\tilde{v}})$. For instance the collection $(i_{\tilde{v}})_{v|p}$ gives an isomorphism
\begin{equation}\label{Gid}
G(\Q_p)=U(F^+\otimes_{\Q}\Q_p) \overset{\sim}{\longrightarrow} \prod_{v|p} \GL_n(F_{\tilde{v}}).
\end{equation}
Similarly $U(F_{\Sigma}^+) \overset{\sim}{\longrightarrow} \prod_{v \in \Sigma }\GL_n(F_{\tilde{v}})$ and analogously for $U(F_{\Sigma_0}^+)$. When there is no risk of confusion we will just write $G$ instead of $G(\Q_p)$. We let $B\subset G$ be the inverse image of the upper-triangular matrices under (\ref{Gid}). In the same fashion $T$ corresponds to the diagonal matrices, and $N$ corresponds to the unipotent radical. Their opposites are denoted $\bar{B}$ and $\bar{N}$. 

Below we will only consider tame levels $K^p\subset G(\A_f^p)$ of the form $K^p=\prod_{v \nmid p} K_v$ where $K_v \subset U(F_v^+)$ is a compact open subgroup which is assumed to be hyperspecial for $v \notin \Sigma$. Accordingly we factor it as
$K^p=K_{\Sigma_0}K^{\Sigma}$ where $K^{\Sigma}=\prod_{v \notin \Sigma} K_v$ is a product of hyperspecials, and $K_{\Sigma_0}=\prod_{v \in \Sigma_0} K_v$. 

\subsection{Automorphic forms}\label{auto}

We work over a fixed finite extension $E/\Q_p$, which we assume is large enough in the sense that
every embedding $F_v^+\hookrightarrow \bar{\Q}_p$ factors through $E$ for all $v|p$. We let $\OO$ denote its valuation ring, $\varpi$ is a choice of uniformizer, and $k=\OO/(\varpi)\simeq \F_q$ is the residue field. We endow $E$ with its normalized absolute value $|\cdot|$ for which $|\varpi|=q^{-1}$. 

For a tame level $K^p\subset G(\A_f^p)$ we introduce the space of $p$-adic automorphic forms on $G(\A)$ as follows (cf. Definition 3.2.3 in \cite{Em06}).
First let
$$
\hat{S}(K^p,\OO)=\CC(G(\Q)\backslash G(\A_f)/K^p,\OO)=\varprojlim_i \CC^{\infty}(G(\Q)\backslash G(\A_f)/K^p,\OO/\varpi^i \OO).
$$
Here $\CC$ is the space of continuous functions, $\CC^{\infty}$ is the space of locally constant functions. Note that the space of locally constant functions in 
$\hat{S}(K^p,\OO)$ is $\varpi$-adically dense, so alternatively
$$
\hat{S}(K^p,\OO)=\CC^{\infty}(G(\Q)\backslash G(\A_f)/K^p,\OO)^{\wedge}=\varprojlim_i \CC^{\infty}(G(\Q)\backslash G(\A_f)/K^p,\OO)\otimes_{\OO} \OO/\varpi^i \OO.
$$
These two viewpoints amount to thinking of $\hat{S}(K^p,\OO)$ as $\tilde{H}^0(K^p)$ or $\hat{H}^0(K^p)$ respectively in the notation of \cite{Em06}, cf. (2.1.1) and Corollary 2.2.25 there. The reduction modulo $\varpi$ is the space of mod $p$ modular forms on $G(\A)$,
$$
S(K^p,k)=\CC^{\infty}(G(\Q)\backslash G(\A_f)/K^p,k) \simeq \hat{S}(K^p,\OO)/\varpi\hat{S}(K^p,\OO),
$$
which is an admissible (smooth) $k[G]$-module with $G=G(\Q_p)$ acting via right translations. Thus $\hat{S}(K^p,\OO)$ is a $\varpi$-adically admissible $G$-representation over $\OO$, i.e. an object of $\Mod_G^{\varpi-\adm}(\OO)$ (cf. Definition 2.4.7 in \cite{Em10}). Since it is clearly flat over $\OO$, it is the unit ball of a Banach representation
$$
\hat{S}(K^p,E)=\hat{S}(K^p,\OO)[1/p]=\CC(G(\Q)\backslash G(\A_f)/K^p,E).
$$
Here we equip the right-hand side with the supremum norm $\|f\|=\sup_{g \in G(\A_f)}|f(g)|$, and $\hat{S}(K^p,E)$ thus becomes an object of the category 
$\Ban_G(E)^{\leq 1}$ of Banach $E$-spaces $(H,\|\cdot\|)$ for which $\|H\|\subset |E|$ endowed with an isometric $G$-action. 
$\hat{S}(K^p,E)$ is dubbed the space of $p$-adic automorphic forms on $G(\A)$.

The connection to classical modular forms is through locally algebraic vectors as we now explain. Let $V$ be an absolutely irreducible algebraic representation of $G\times_{\Q} E$. Thus $V$ is a finite-dimensional $E$-vector space with an action of $G(E)$, which we restrict to $G(\Q_p)$. If $K_p\subset G(\Q_p)$ is a compact open subgroup we let it act on $V$ and consider
$$
S_V(K_pK^p,E)=\Hom_{K_p}(V, \hat{S}(K^p,E)).
$$
If we assume $E$ is large enough that $\End_G(V)=E$, the space of $V$-locally algebraic vectors in $\hat{S}(K^p,E)$ can be defined as 
the image of the natural map
$$
\varinjlim_{K_p}V \otimes_E S_V(K_pK^p,E) \overset{\sim}{\longrightarrow}  \hat{S}(K^p,E)^{V-\alg} \hookrightarrow \hat{S}(K^p,E)
$$
(cf. Proposition 4.2.4 in \cite{Em11}). Then the space of all locally algebraic vectors decomposes as a direct sum $\hat{S}(K^p,E)^{\alg}=\bigoplus_V \hat{S}(K^p,E)^{V-\alg}$. Letting $\tilde{V}$ denote the contragredient representation, one easily identifies $S_V(K_pK^p,E)$ with the space of (necessarily continuous) functions
$$
f: G(\Q)\backslash G(\A_f)/K^p \longrightarrow \tilde{V}, \y \y f(gk)=k^{-1}f(g) \y \y \forall k \in K_p.
$$
In turn, considering the function $h(g)=gf(g)$ identifies it with the space of right $K_pK^p$-invariant functions
$h:G(\A_f)\rightarrow \tilde{V}$ such that $h(\gamma g)=\gamma h(g)$ for all $\gamma \in G(\Q)$. If we complexify this space along an embedding 
$\iota: E \hookrightarrow \C$ we obtain vector-valued automorphic forms. Thus we arrive at the decomposition
\begin{equation}\label{aut}
S_V(K_pK^p,E)\otimes_{E,\iota}\C\simeq \bigoplus_{\pi} m_G(\pi) \cdot \pi_p^{K_p} \otimes (\pi_f^p)^{K^p}
\end{equation}
with $\pi$ running over automorphic representations of $G(\A)$ with $\pi_{\infty}\simeq V\otimes_{E,\iota}\C$. It is even known by now that all
$m_G(\pi)=1$, cf. \cite{Mok15} and 'the main global theorem' \cite[Thm. 1.7.1, p. 89]{KMSW} (both based on the symplectic/orthogonal case \cite{Art13}). Multiplicity one will be used below in Lemma \ref{unique}.

\begin{rem}\label{mult}
For full disclosure we will only use multiplicity one for representations $\pi$ whose base change $\Pi=\text{BC}_{F/F^+}(\pi)$ to $\GL_n(\A_F)$ is cuspidal (cf. the proof of Lemma \ref{unique} below). Since $\Pi_{\infty}$ is $V$-cohomological the Ramanujan conjecture holds in this case, i.e. $\Pi$ is tempered. Therefore the packets in \cite[Thm. 1.7.1]{KMSW}
do not overlap and consists of irreducible representations; in particular $m_G(\pi)=1$. Some of the authors of \cite{KMSW} have informed us that multiplicity one even holds for non-tempered representations $\pi$, the point being that the groups $S^\natural_{\psi_v}$ in loc. cit. are abelian. As mentioned in the introduction to loc. cit. the non-tempered case is the topic of a sequel. 
\end{rem}

\subsection{Hecke algebras}\label{hecke}

At each $v\nmid p$ we consider the Hecke algebra $\HH(U(F_v^+),K_v)$ of $K_v$-biinvariant compactly supported functions 
$\phi:U(F_v^+)\rightarrow \OO$ (with $K_v$-normalized convolution). The characteristic functions of double cosets $[K_v\gamma_vK_v]$ form an $\OO$-basis. 

Suppose $v$ splits in $F$ and $K_v$ is hyperspecial. Choose a place $w|v$ and an isomorphism $i_w$ which restricts to $i_w: K_v \overset{\sim}{\longrightarrow} \GL_n(\OO_{F_w})$. Then we identify $\HH(U(F_v^+),K_v)$ with the spherical Hecke algebra for $\GL_n(F_w)$. We let $\gamma_{w,j}\in U(F_v^+)$ denote the element corresponding to 
$$
i_w(\gamma_{w,j})=\diag(\underbrace{\varpi_{F_w},\ldots,\varpi_{F_w}}_{j},1,\ldots,1).
$$
Then let $T_{w,j}=[K_v\gamma_{w,j} K_v]$ be the standard Hecke operators; $\HH(U(F_v^+),K_v)=\OO[T_{w,1},\ldots, T_{w,n}^{\pm1}]$.

For a tame level $K^p$ as above, the full Hecke algebra
$$
\HH(G(\A_f^p),K^p)=\bigotimes_{v\nmid p} \HH(U(F_v^+),K_v)
$$
acts on $\hat{S}(K^p,E)$ by norm-decreasing morphisms, and hence preserves the unit ball $\hat{S}(K^p,\OO)$. This induces actions on $S(K^p,k)$ and $S_V(K_pK^p,E)$ as well
given by the usual double coset operators. Let
$$
\HH(K_{\Sigma_0})=\bigotimes_{v\in \Sigma_0} \HH(U(F_v^+),K_v), \y \y \y \HH_s(K^{\Sigma})={\bigotimes}_{\text{$v\notin \Sigma$ split}} \HH(U(F_v^+),K_v)
$$
be the subalgebras of $\HH(G(\A_f^p),K^p)$ generated by Hecke operators at $v \in \Sigma_0$, respectively $T_{w,1},\ldots, T_{w,n}^{\pm 1}$ for $v \notin \Sigma$ {\it{split}} in $F$ and $w|v$ (the subscript $s$ is for 'split'). In what follows we ignore the Hecke action at the non-split places $v \notin \Sigma$. Note that $\HH_s(K^{\Sigma})$ is commutative, but of course $\HH(K_{\Sigma_0})$ need not be. 

We define the Hecke polynomial $P_w(X)\in \HH_s(K^{\Sigma})[X]$ to be
$$
P_w(X)=X^n+\cdots+(-1)^j (Nw)^{j(j-1)/2}T_{w,j} X^{n-j}+\cdots+(-1)^n (Nw)^{n(n-1)/2}T_{w,n}
$$
where $Nw$ is the size of the residue field $\OO_{F_w}/(\varpi_{F_w})$.

We denote by $\T_V(K_pK^p,\OO)$ the subalgebra of $\End  \big(S_V(K_pK^p,E)\big)$ generated by the operators $\HH_s(K^{\Sigma})$. This is reduced and finite over $\OO$. In case $V$ is the trivial representation we write $\T_0(K_pK^p,\OO)$. As $K_p$ shrinks there are surjective transition maps between these (given by restriction) and we let 
$$
\hat{\T}(K^p,\OO)=\varprojlim_{K_p} \T_0(K_pK^p,\OO),
$$
equipped with the projective limit topology (each term being endowed with the $\varpi$-adic topology). We refer to it as the 'big' Hecke algebra. $\hat{\T}(K^p,\OO)$ clearly acts faithfully on $\hat{S}(K^p,E)$ and one can easily show that the natural map $\HH_s(K^{\Sigma}) \rightarrow \hat{\T}(K^p,\OO)$ has dense image, cf. the discussion in \cite[5.2]{Em11b}. 

A maximal ideal $\m \subset \HH_s(K^{\Sigma})$ is called automorphic (of tame level $K^p$) if it arises as the pullback of a maximal ideal in some $\T_V(K_pK^p,\OO)$.
Shrinking $K_p$ if necessary we may assume it is pro-$p$, in which case we may take $V$ to be trivial ('Shimura's principle'). In particular there are only finitely many such $\m$,
and we interchangeably view them as maximal ideals of $\hat{\T}(K^p,\OO)$ (and use the same notation), which thus factors as a finite product of complete local $\OO$-algebras
$$
\hat{\T}(K^p,\OO)=\prod_{\m} \hat{\T}(K^p,\OO)_{\m}.
$$
Correspondingly we have a decomposition $\hat{S}(K^p,E)=\bigoplus_{\m} \hat{S}(K^p,E)_{\m}$, and similarly for $\hat{S}(K^p,\OO)$.
This direct sum is clearly preserved by $\HH(K_{\Sigma_0})$.

\subsection{Galois representations}\label{gal}

If $R$ is an $\OO$-algebra, and $r:\Gal_F \rightarrow \GL_n(R)$ is an arbitrary representation which is unramified at all places $w$ of $F$ lying above a split $v \notin \Sigma$, we associate the eigensystem $\theta_r: \HH_s(K^{\Sigma}) \rightarrow R$ determined by
$$
\det(X-r(\Frob_w))=\theta_r(P_w(X)) \in R[X]
$$
for all such $w$. Here $\Frob_w$ denotes a geometric Frobenius. (Note that the coefficients of the polynomial determine $\theta_r(T_{w,j})$ since $Nw \in \OO^{\times}$; 
and $\theta_r(T_{w,n})\in R^{\times}$.) We say $r$ is automorphic (for $G$) if $\theta_r$ factors through one of the quotients $\T_V(K_pK^p,\OO)$. 

When $R=\OO$ this means $r$ is  
associated with one of the automorphic representations $\pi$ contributing to (\ref{aut}) in the sense that $T_{w,j}$ acts on $\pi_v^{K_v}$ by scaling by $\iota(\theta_r(T_{w,j}))$ for all $w|v \notin \Sigma$ as above. Conversely, it is now known that to any such $\pi$ (and a choice of isomorphism $\iota: \bar{\Q}_p \overset{\sim}{\longrightarrow} \C$) one can attach a unique semisimple Galois representation $r_{\pi,\iota}: \Gal_F \rightarrow \GL_n(\bar{\Q}_p)$ with that property, cf. \cite[Theorem 6.5]{Tho12} for a nice summary.
It is polarized, meaning that $r_{\pi,\iota}^{\vee}\simeq r_{\pi,\iota}^c \otimes \epsilon^{n-1}$ where $\epsilon$ is the cyclotomic character, and one can explicitly write down its Hodge-Tate weights in terms of $V$. 

When $R=k$ we let $\m_r=\ker(\theta_r)$ be the corresponding maximal ideal of $\HH_s(K^{\Sigma})$. Then $r$ is automorphic precisely when $\m_r$ is automorphic, in which case we tacitly view it as a maximal ideal of $\T_V(K_pK^p,\OO)$ (with residue field $k$) for suitable $V$ and $K_p$. In the other direction, starting from a maximal ideal 
$\m$ in $\T_V(K_pK^p,\OO)$ (whose residue field is necessarily a finite extension of $k$) one can attach a unique semisimple representation
$$
\bar{r}_{\m}:\Gal_F \longrightarrow \GL_n(\T_V(K_pK^p,\OO)/\m)
$$
such that $\theta_{\bar{r}_{\m}}(T_{w,j})=T_{w,j}+\m$ (and which is polarized), cf. \cite[Prop. 6.6]{Tho12}. We say $\m$ is non-Eisenstein if $\bar{r}_{\m}$ is absolutely irreducible.
Under this hypothesis $\bar{r}_{\m}$ admits a (polarized) lift 
$$
r_{\m}: \Gal_F \longrightarrow \GL_n(\T_V(K_pK^p,\OO)_{\m})
$$
with the property that $\theta_{r_{\m}}(T_{w,j})=T_{w,j}$; it is unique up to conjugation, cf. \cite[Prop. 6.7]{Tho12}, and gives a well-defined deformation of $\bar{r}_{\m}$. If we let $K_p$ shrink to a pro-$p$ subgroup we may take $V$ to be trivial, i.e. $\m\subset \T_{{\bf{1}}}(K_pK^p,\OO)$. Passing to the inverse limit yields a lift of $\bar{r}_{\m}$ with coefficients in $\hat{\T}(K^p,\OO)_{\m}$ which we will denote by $\hat{r}_{\m}$. Throughout \cite{Tho12} it is assumed that $p>2$; we adopt that hypothesis here.

All the representations discussed above ($r_{\pi,\iota}$, $\bar{r}_{\m}$, $r_{\m}$ etc.) extend\footnote{Once a choice of $\gamma_0 \in \Gal_{F^+}-\Gal_F$ is made, cf. \cite[Lem. 2.1.4]{CHT08}. See also Prop. 3.4.4 therein.} to continuous homomorphisms $\Gal_{F^+}\rightarrow \mathcal{G}_n(R)$ 
for various $R$, where $\mathcal{G}_n$ is the group scheme (over $\Z$) defined as a semi-direct product $\{1,j\}\ltimes (\GL_n \times \GL_1)$, cf. \cite[Def. 2.1]{Tho12}.
We let $\nu: \mathcal{G}_n\rightarrow \GL_1$ be the natural projection. Thus $\nu \circ \bar{r}_{\m}=\epsilon^{1-n}\delta_{F/F^+}^{\mu_{\m}}$ (and similarly for $r_{\m}$) where $\delta_{F/F^+}$ is the non-trivial quadratic character of $\Gal(F/F^+)$ and $\mu_{\m}\in \{0,1\}$ is determined by the congruence $\mu_{\m}\equiv n$ mod $2$ (cf. \cite[Thm. 3.5.1]{CHT08} and \cite[Thm. 1.2]{BC11}).

\subsection{Deformations}\label{def}

Now start with $\bar{r}: \Gal_{F^+}\rightarrow \mathcal{G}_n(k)$ such that its restriction $\bar{r}: \Gal_F \rightarrow \GL_n(k)$ is absolutely irreducible and automorphic, with corresponding maximal ideal $\m=\m_{\bar{r}}$, and $\nu \circ \bar{r}=\epsilon^{1-n}\delta_{F/F^+}^{\mu_{\m}}$. In particular $\bar{r}$ is unramified outside $\Sigma$.

We consider lifts and deformations of $\bar{r}$ to rings in $\mathcal{C}_{\OO}$, the category of complete local Noetherian $\OO$-algebras $R$ with residue field
$k \overset{\sim}{\longrightarrow} R/\m_R$, cf. \cite[Def. 3.1]{Tho12}. Recall that a lift is a homomorphism $r: \Gal_{F^+}\rightarrow \mathcal{G}_n(R)$ such that
$r$ reduces to $\bar{r}$ mod $\m_R$, and $\nu \circ r=\epsilon^{1-n}\delta_{F/F^+}^{\mu_{\m}}$ (thought of as taking values in $R^{\times}$). A deformation is a
$(1+M_n(R))$-conjugacy class of lifts.

For each $v \in \Sigma$ consider the restriction $\bar{r}_{\tilde{v}}=\bar{r}|_{\Gal_{F_{\tilde{v}}}}$ and its universal lifting ring $R_{\bar{r}_{\tilde{v}}}^{\square}$. Following \cite{Tho12} we let $R_{\bar{r}_{\tilde{v}}}^{\bar{\square}}$ denote its maximal reduced $p$-torsion free quotient, and consider the deformation problem
$$
\mathcal{S}=\big( F/F^+, \Sigma, \tilde{\Sigma}, \OO, \bar{r}, \epsilon^{1-n}\delta_{F/F^+}^{\mu_{\m}}, \{R_{\bar{r}_{\tilde{v}}}^{\bar{\square}}\}_{v\in \Sigma} \big).
$$
The functor $\text{Def}_{\mathcal{S}}$ of deformations of type $\mathcal{S}$ is then represented by an object $R_{\mathcal{S}}^{univ}$ of $\mathcal{C}_{\OO}$, cf. 
\cite[Prop. 3.4]{Tho12} or \cite[Prop. 2.2.9]{CHT08}. In what follows we will simply write $R_{\bar{r}}$ instead of $R_{\mathcal{S}}^{univ}$, and keep in mind the underlying deformation problem $\mathcal{S}$. Similarly, $R_{\bar{r}}^{\square}$ is the universal lifting ring of type $\mathcal{S}$ (which is denoted by $R_{\mathcal{S}}^{\square}$ in \cite[Prop. 3.4]{Tho12}). Note that $R_{\bar{r}}^{\square}$ is a power series $\OO$-algebra in $|\Sigma|n^2$ variables over $R_{\bar{r}}$ (\cite[Prop. 2.2.9]{CHT08}); a fact we will not use in this paper.

The universal automorphic deformation $r_{\m}$ is of type $\mathcal{S}$, so by universality it arises from a local homomorphism
$$
\psi: R_{\bar{r}}\longrightarrow \T_V(K_pK^p,\OO)_{\m}.
$$
These maps are compatible as we shrink $K_p$. Taking $V$ to be trivial and passing to the inverse limit over $K_p$ we obtain a map $\hat{\psi}: R_{\bar{r}} \rightarrow \hat{\T}(K^p,\OO)_{\m}$ which we use to view $\hat{S}(K^p,E)_{\m}$ as an $R_{\bar{r}}$-module.

\section{Eigenvarieties}\label{ev}

\subsection{Formal schemes and rigid spaces}

In what follows $(-)^{\rig}$ will denote Berthelot's functor (which generalizes Raynaud's construction for topologically finite type formal schemes $\frak{X}$ over $\Spf(\OO)$, cf. \cite{Ray74}). Its basic properties are nicely reviewed in 
\cite[Ch. 7]{dJ95}. The source $\text{FS}_{\OO}$ is the category of locally Noetherian adic formal schemes $\frak{X}$ which are formally of finite type over $\Spf(\OO)$ (i.e., their reduction modulo an ideal of definition is of finite type over $\Spec(k)$); the target $\text{Rig}_E$ is the category of rigid analytic varieties over $E$, cf. Definition 9.3.1/4 in \cite{BGR84}. For example,
$\B=(\Spf \OO\{y\})^{\rig}$ is the closed unit disc (at $0$); $\U=(\Spf \OO[[x]])^{\rig}$ is the open unit disc.  For a general affine formal scheme
$\frak{X}=\Spf(A)$ where 
$$
A=\OO\{y_1,\ldots,y_r\}[[x_1,\ldots,x_s]]/(g_1,\ldots,g_t),
$$ 
$\frak{X}^{\rig}\subset \B^r \times \U^s$ is the closed analytic subvariety cut out by the functions $g_1,\ldots,g_t$, cf. \cite[9.5.2]{BGR84}. In general $\frak{X}^{\rig}$ is obtained by glueing affine pieces as in \cite[7.2]{dJ95}. The construction of $\frak{X}^{\rig}$ in the affine case is actually completely canonical and free from coordinates: If $I \subset A$ is the largest ideal of definition, $A[I^n/\varpi]$ is the subring of $A \otimes_{\OO}E$ generated by $A$ and all $i/\varpi$ with $i \in I^n$. Let $A[I^n/\varpi]^{\wedge}$ be its $I$-adic completion
(equivalently, its $\varpi$-adic completion, see the proof of \cite[Lem. 7.1.2]{dJ95}). Then $A[I^n/\varpi]^{\wedge}\otimes_{\OO} E$ is an affinoid $E$-algebra and there is an admissible covering
$$
\frak{X}^{\rig}=\Spf(A)^{\rig}=\bigcup_{n=1}^{\infty} \text{Sp}(A[I^n/\varpi]^{\wedge}\otimes_{\OO} E).
$$
In particular $A^{\rig}:=\OO(\Spf(A)^{\rig})=\varprojlim_n A[I^n/\varpi]^{\wedge}\otimes_{\OO} E$. The natural map $A \otimes_{\OO} E \rightarrow A^{\rig}$ factors through the ring of bounded functions on $\Spf(A)^{\rig}$; the image of $A$ lies in $\OO^0(\Spf(A)^{\rig})$, the functions whose absolute value is bounded by $1$, cf. \cite[7.1.8]{dJ95}.

\subsection{Deformation space}\label{defspace}

We let $X_{\bar{r}}=\Spf(R_{\bar{r}})^{\rig}$ (a subvariety of $\U^s$ for some $s$). For a point $x \in X_{\bar{r}}$ we let $\kappa(x)$ denote its residue field, which is a finite extension of $E$, and let $\kappa(x)^0$ be its valuation ring; an $\OO$-algebra with finite residue field $k(x)$. Note the different meanings of $\kappa(x)$ and $k(x)$.
The evaluation map $R_{\bar{r}}\rightarrow \OO^0(X_{\bar{r}})\rightarrow \kappa(x)^0$
corresponds to a deformation
$$
r_x: \Gal_{F^+}\longrightarrow \mathcal{G}_n(\kappa(x)^0)
$$
of $\bar{r}\otimes_k k(x)$. (We tacitly choose a representative $r_x$ in the conjugacy class of lifts.) We let $\p_x=\ker(R_{\bar{r}}\rightarrow \kappa(x)^0)$ be the prime ideal of $R_{\bar{r}}$ corresponding to $x$, cf. the bijection in \cite[Lem. 7.1.9]{dJ95}.
We will often assume for notational simplicity that $x$ is $E$-rational, in which case $\kappa(x)=E$ and $k(x)=k$; so that $r_x$ is a deformation of $\bar{r}$ over $\kappa(x)^0=\OO$. 

\subsection{Character and weight space}\label{charspace}

Recall our choice of torus $T \subset G(\Q_p)$, and let $T_0$ be its maximal compact subgroup. Upon choosing uniformizers $\{\varpi_{F_{\tilde{v}}}\}_{v|p}$ we have an isomorphism
$T\simeq T_0 \times \Z^{n|\Sigma_p|}$ of topological groups. Moreover,
$$
T_0\simeq \prod_{v|p} (\OO_{F_{\tilde{v}}}^{\times})^n\simeq \big(\underbrace{\prod_{v|p}\mu_{\infty}(F_{\tilde{v}})^n}_{\mu} \big) \times \Z_p^{n[F^+:\Q]}.
$$
Let $\hat{T}:=\WW \times (\G_m^{\rig})^{n|\Sigma_p|}$ where $\WW:=\big(\Spf (\OO[[T_0]])\big)^{\rig}$. The weight space $\WW$ is isomorphic to $|\mu|$ copies of the open unit ball $\U^{n[F^+:\Q]}$. From a more functorial point of view $\hat{T}$ represents the functor which takes an affinoid $E$-algebra to 
the set $\Hom_{\text{cont}}(T,A^{\times})$, and similarly for $\WW$ and $T_0$. See \cite[Prop. 6.4.5]{Em11}. Thus $\hat{T}$ carries a universal continuous character 
$\delta^{\text{univ}}: T\rightarrow \OO(\hat{T})^{\times}$ which restricts to a character $T_0 \rightarrow \OO^0(\WW)^{\times}$ via the canonical morphism $\hat{T}\rightarrow \WW$.
Henceforth we identify points of $\hat{T}$ with continuous characters $\delta: T \rightarrow \kappa(\delta)^{\times}$ for varying finite extensions $\kappa(\delta)$ of $E$ (and analogously for $\WW$). 

\subsection{Definition of the eigenvariety}

We follow \cite[4.1]{BHS16} in defining the eigenvariety $Y(K^p,\bar{r})$ as the support of a certain coherent sheaf $\M=\M_{K^p}$ on $X_{\bar{r}}\times \hat{T}$. This is basically also the approach taken in section (2.3) of \cite{Em06}, except there $X_{\bar{r}}$ is replaced by $\Spec$ of a certain Hecke algebra. We define $\M$ as follows.

Let $(-)^{\an}$ be the functor from \cite[Thm. 7.1]{ST03}. It takes an object $H$ of $\Ban_G^{\adm}(E)$ to the dense subspace $H^{\an}$ of locally analytic vectors. $H^{an}$ is a locally analytic $G$-representation (over $E$) of compact type whose strong dual $(H^{an})'$ is a coadmissible $D(G,E)$-module, cf. \cite[p. 176]{ST03}.

We take $H=\hat{S}(K^p,E)_{\m}$ and arrive at an admissible locally analytic $G$-representation $\hat{S}(K^p,E)_{\m}^{\an}$ which we feed into the Jacquet functor $J_B$ defined in \cite[Def. 3.4.5]{Em06b}. By Theorem 0.5 of loc. cit. this yields an {\it{essentially}} admissible locally analytic $T$-representation $J_B(\hat{S}(K^p,E)_{\m}^{\an})$. See 
\cite[Def. 6.4.9]{Em11} for the notion of essentially admissible (the difference with admissibility lies in incorporating the action of the center $Z$, or rather viewing the strong dual as a module over $\OO(\hat{Z})\hat{\otimes} D(G,E)$). 

We recall \cite[Prop. 2.3.2]{Em06}: If $\FF$ is a coherent sheaf on $\hat{T}$, cf. \cite[Def. 9.4.3/1]{BGR84}, its global sections $\Gamma(\hat{T},\FF)$ is a coadmissible $\OO(\hat{T})$-module. Moreover, the functor $\FF \rightsquigarrow \Gamma(\hat{T},\FF)$ is an equivalence of categories (since $\hat{T}$ is quasi-Stein). Note that $\Gamma(\hat{T},\FF)$ and its strong dual both acquire a $T$-action via $\delta^{\text{univ}}$. Altogether the functor $\FF \rightsquigarrow \Gamma(\hat{T},\FF)'$ sets up an anti-equivalence of categories between coherent sheaves on $\hat{T}$ and essentially admissible locally analytic $T$-representations (over $E$). 

As pointed out at the end of section \ref{def}, $\hat{S}(K^p,E)_{\m}$ is an $R_{\bar{r}}$-module via $\hat{\psi}$, and the $G$-action is clearly $R_{\bar{r}}$-linear. Thus 
$J_B(\hat{S}(K^p,E)_{\m}^{\an})$ inherits an $R_{\bar{r}}$-module structure. By suitably modifying the remarks of the preceding paragraph (as in section 3.1 of \cite{BHS16} where they define and study locally $R_{\bar{r}}$-analytic vectors, cf. Def. 3.2 in loc. cit.) one finds that there is a coherent sheaf $\M=\M_{K^p}$ on $X_{\bar{r}}\times \hat{T}$ for which
$$
J_B(\hat{S}(K^p,E)_{\m}^{\an})\simeq \Gamma(X_{\bar{r}}\times \hat{T},\M)'.
$$ 
The {\it{eigenvariety}} is then defined as the (schematic) support of $\M$, cf. \cite[Prop. 9.5.2/4]{BGR84}. I.e.,
$$
Y(K^p,\bar{r}):=\text{supp} (\M)=\{y=(x,\delta): \M_y\neq 0\} \subset X_{\bar{r}}\times \hat{T}.
$$
Thus $Y(K^p,\bar{r})$ is an analytic subset of $X_{\bar{r}}\times \hat{T}$ with structure sheaf $\OO_{X_{\bar{r}}\times \hat{T}}/\mathcal{I}$, where $\mathcal{I}$ is the ideal sheaf of annihilators of $\M$. That is $\mathcal{I}(U)=\text{Ann}_{\OO(U)}\Gamma(U,\M)$ for admissible open $U$. One can show that $Y(K^p,\bar{r})$ is reduced, cf. part (3) of Lemma \ref{ingredients} below for precise references.

The fiber $\M_y=\big(\varinjlim_{U \ni y} \Gamma(U,\M)\big)\otimes_{\OO_{Y(K^p,\bar{r}),y}}\kappa(y)$ 
is finite-dimensional over $\kappa(y)$. Suppose $\kappa(y)\simeq E$ solely to simplify the notation. Then the full $E$-linear dual 
$\M_y'=\Hom_E(\M_y,E)$ has the following useful description.

\begin{lem}\label{fib}
Let $y=(x,\delta)\in (X_{\bar{r}}\times \hat{T})(E)$ be an $E$-rational point. Then there is an isomorphism
\begin{equation}\label{fiber}
\M_y'\simeq J_B^{\delta}(\hat{S}(K^p,E)_{\m}[\p_x]^{\an}).
\end{equation}
(Here $J_B^{\delta}$ means the $\delta$-eigenspace of $J_B$, and $[\p_x]$ means taking $\p_x$-torsion.)
\end{lem}

\begin{proof}
First, since $X_{\bar{r}}\times \hat{T}$ is quasi-Stein, $\M_y$ is the largest quotient of $\Gamma(X_{\bar{r}}\times \hat{T},\M)$ which is annihilated by $\p_x$ and on which $T$ acts via $\delta$, cf. \cite[5.4]{BHS16}. Thus $\M_y'$ is the largest subspace of $J_B(\hat{S}(K^p,E)_{\m}^{\an})$ with the same properties, i.e. $J_B^{\delta}(\hat{S}(K^p,E)_{\m}^{\an})[\p_x]$, as observed in Proposition 2.3.3 (iii) of \cite{Em06}. Now,
$$
J_B^{\delta}(\hat{S}(K^p,E)_{\m}^{\an})[\p_x]=J_B^{\delta}(\hat{S}(K^p,E)_{\m}[\p_x]^{\an})
$$
as follows easily from the exactness of $(-)^{\an}$ and the left-exactness of $J_B$ (using that $\p_x$ is finitely generated to reduce to the principal case by induction on the number of generators), cf. the proof of \cite[Prop. 3.7]{BHS16}.
\end{proof}

The space in (\ref{fiber}) can be made more explicit: Choose a compact open subgroup $N_0\subset N$ and introduce the monoid $T^+=\{t \in T: tN_0t^{-1}\subset N_0\}$. Then by \cite[Prop. 3.4.9]{Em06b},
$$
J_B^{\delta}(\hat{S}(K^p,E)_{\m}[\p_x]^{\an}) \simeq \big(\hat{S}(K^p,E)_{\m}[\p_x]^{\an}\big)^{N_0,T^+=\delta}
$$
where $T^+$ acts by double coset operators $[N_0tN_0]$ on the space on the right. Observe that $y$ lies on the eigenvariety $Y(K^p,\bar{r})$ precisely when the above space $\M_y'$ is nonzero.

Note that the Hecke algebra $\HH(K_{\Sigma_0})$ acts on $J_B(\hat{S}(K^p,E)_{\m}^{\an})$, and therefore on $\M$ and its fibers $\M_y$ (on the {\it{right}} since we are taking duals). The isomorphism (\ref{fiber}) is 
$\HH(K_{\Sigma_0})$-equivariant, and our first goal is to describe $\M_y'$ as a $\HH(K_{\Sigma_0})$-module.

\subsection{Classical points}

We say that a point $y=(x,\delta)\in Y(K^p,\bar{r})(E)$ is {\it{classical}} (of weight $V$) if the following conditions hold (cf. \cite[Def. 3.14]{BHS16} or the paragraph before \cite[Def. 0.6]{Em06}):

\begin{itemize}
\item[(1)] $\delta=\delta_{\alg}\delta_{\sm}$, where $\delta_{\alg}$ is an algebraic character which is dominant relative to $B$ (i.e., obtained from an element of $X^*(T\times_{\Q}E)^+$ by restriction to $T$), and $\delta_{\sm}$ is a smooth character of $T$. In this case let $V$ denote the irreducible algebraic representation of $G \times_{\Q} E$ of highest weight $\delta_{\alg}$.  
\item[(2)] There exists an automorphic representation $\pi$ of $G(\A)$ such that
\begin{itemize}
\item[(a)] $(\pi_f^p)^{K^p}\neq 0$ and the $\HH_s(K^{\Sigma})$-action on this space is given by the eigensystem $\iota\circ \theta_{r_x}$,
\item[(b)] $\pi_{\infty}\simeq V \otimes_{E,\iota}\C$,
\item[(c)] $\pi_p$ is a quotient of $\Ind_{\bar{B}}^G(\delta_{\sm}\delta_B^{-1})$.
\end{itemize}
\end{itemize}
These points comprise the subset $Y(K^p,\bar{r})_{\cl}$. Note that condition (a) is equivalent to the isomorphism $r_x \simeq r_{\pi,\iota}$ (both sides are irreducible since $r_x$ is a lift of $\bar{r}$). In (c) $\delta_B$ denotes the modulus character of $B$; the reason we include it in condition (c) will become apparent in the proof of Prop. \ref{ncr} below. 

\begin{lem}\label{unique}
There is at most one automorphic $\pi$ satisfying (a)--(c) above; and $m_G(\pi)=1$.
\end{lem}

\begin{proof}
Let $\Pi=\BC_{F/F^+}(\pi)$ be a (strong) base change of $\pi$ to $\GL_n(\A_F)$, where we view $\pi$ as a representation of $U(\A_{F^+})=G(\A)$. For its existence see \cite[5.3]{Lab11}.
Note that $\Pi$ is cuspidal since $r_{\pi,\iota}$
is irreducible. In particular  $\Pi$ is globally generic, hence locally generic. By local-global compatibility, cf. \cite{BGGT}, \cite{BGGT2}, and \cite{Car14} for places $w|p$; \cite{TY07} and
\cite{Shi11} for places $w \nmid p$,
$$
\iota\WD(r_{\pi,\iota}|_{\Gal_{F_{w}}})^{F-ss}\simeq \rec(\Pi_w \otimes |\det|^{(1-n)/2})
$$
for {\it{all}} finite places $w$ of $F$, with the local Langlands correspondence $\rec(\cdot)$ normalized as in \cite{HT01}. This shows that $\Pi_w$ is completely determined by $r_x$ at all finite places $w$. Moreover, we have $\Pi_w=\BC_{w|v}(\pi_v)$ whenever the local base change on the right is defined, i.e. when either $v$ splits or $\pi_v$ is unramified. 
Our assumption that $\Sigma$ consists of split places guarantees that $\BC_{w|v}(\pi_v)$ makes sense locally everywhere. Furthermore, unramified local base change is injective according to \cite[Cor. 4.2]{Min11}. We conclude that $\pi_f$ is determined by $r_x$, and $\pi_{\infty}\simeq V \otimes_{E,\iota}\C$. Thus $\pi$ is unique. Multiplicity one was noted earlier at the end of section \ref{auto} above, cf. Remark \ref{mult}.
\end{proof}

\section{The case of classical points of non-critical slope}\label{noncrit}

Each point $x \in X_{\bar{r}}$ carries a Galois representation $r_x: \Gal_F \rightarrow \GL_n(\kappa(x))$ which we restrict to the various decomposition groups
$\Gal_{F_{\tilde{v}}}$ for $v \in \Sigma$. When $v \in \Sigma_0$ there is a corresponding Weil-Deligne representation, cf. section (4.2) in \cite{Tat79}, and we let
$\pi_{x,v}$ be the representation of $U(F_v^+)$ (over $\kappa(x)$) such that
\begin{equation}\label{pix}
\WD(r_x|_{\Gal_{F_{\tilde{v}}}})^{F-ss}\simeq \rec(\BC_{\tilde{v}|v}(\pi_{x,v}) \otimes |\det|^{(1-n)/2})
\end{equation}
Note that the local base change $\BC_{\tilde{v}|v}(\pi_{x,v})$ is just $\pi_{x,v}$ thought of as a representation of $\GL_n(F_{\tilde{v}})$ via the isomorphism 
$i_{\tilde{v}}: U(F_v^+) \overset{\sim}{\longrightarrow} \GL_n(F_{\tilde{v}})$. We emphasize that $\pi_{x,v}$ is defined even for {\it{non}}-classical points on the eigenvariety. 
If $y=(x,\delta)$ happens to be classical, $\pi_{x,v}\otimes_{E,\iota}\C \simeq \pi_v$ where $\pi$ is the automorphic representation in Lemma \ref{unique}. Below we relate 
$\otimes_{v\in \Sigma_0} \pi_{x,v}$ to the fiber $\M_y'$.

\begin{prop}\label{ncr}
Let $y=(x,\delta)\in Y(K^p,\bar{r})(E)$ be a classical point. Then there exists an embedding of $\HH(K_{\Sigma_0})$-modules
${\bigotimes}_{v\in \Sigma_0}\pi_{x,v}^{K_v} \hookrightarrow \M_y'$ which is an isomorphism if $\delta$ is of non-critical slope, cf. \cite[Def. 4.4.3]{Em06b} (which is summarized below).

\end{prop}

\begin{proof}
According to (0.14) in \cite{Em06b} there is a closed embedding
$$
J_B\big((\hat{S}(K^p,E)_{\m}[\p_x]^{\an})^{V-\alg}\big)\hookrightarrow J_B\big(\hat{S}(K^p,E)_{\m}[\p_x]^{\an} \big)^{V^N-\alg}.
$$
Note that $V^N\simeq \delta_{\alg}$ so after passing to $\delta$-eigenspaces we get a closed embedding
\begin{equation}\label{emb}
J_B^{\delta}\big((\hat{S}(K^p,E)_{\m}[\p_x]^{\an})^{V-\alg}\big)\hookrightarrow J_B^{\delta}\big(\hat{S}(K^p,E)_{\m}[\p_x]^{\an} \big).
\end{equation}
The target is exactly $\M_y'$ by (\ref{fiber}). On the other hand
$$
(\hat{S}(K^p,E)_{\m}[\p_x]^{\an})^{V-\alg}\simeq \bigoplus_{\pi} (V\otimes_E \pi_p) \otimes_E (\pi_f^p)^{K^p}
$$
with $\pi$ running over automorphic representations of $G(\A)$ over $E$ with $\pi_{\infty}\simeq V$ and such that $\theta_{r_x}$ gives the action of 
$\HH_s(K^{\Sigma})$ on $(\pi_f^p)^{K^p}$. As noted in Lemma \ref{unique} there is precisely one such $\pi$ which we will denote by $\pi_x$ throughout this proof (consistent with the notation $\pi_{x,v}$ introduced above). Note that $\otimes_{v\notin \Sigma}\pi_{x,v}^{K_v}$ is a line so
$$
(\hat{S}(K^p,E)_{\m}[\p_x]^{\an})^{V-\alg}\simeq (V\otimes_E \pi_{x,p})\otimes_E \big({\bigotimes}_{v\in \Sigma_0}\pi_{x,v}^{K_v}\big).
$$
Since $J_B$ is compatible with the classical Jacquet functor, cf. \cite[Prop. 4.3.6]{Em06b}, we identify the source of (\ref{emb}) with
$$
(V^N\otimes_E (\pi_{x,p})_N)^{T=\delta}\otimes_E \big({\bigotimes}_{v\in \Sigma_0}\pi_{x,v}^{K_v}\big).
$$
Now $V^N\simeq \delta_{\alg}$ is one-dimensional, and so is $(\pi_{x,p})_N^{T=\delta_{\sm}}$. Indeed, by Bernstein second adjointness, 
$$
(\pi_{x,p})_N^{T=\delta_{\sm}}\simeq \Hom_G\big(\Ind_{\bar{B}}^G (\delta_{\sm}\delta_B^{-1}), \pi_{x,p} \big).
$$
The right-hand side is nonzero by condition (c) above, and in fact it is a line since $\Ind_{\bar{B}}^G (\delta_{\sm}\delta_B^{-1})$ has a unique generic constituent 
(namely $\pi_{x,p}$, cf. the proof of Lemma \ref{unique}) which occurs with multiplicity one; this follows from the theory of derivatives \cite[Ch. 4]{BZ77}. From this observation we immediately infer that $\Hom_G\big(\tilde{\pi}_{x,p}, \Ind_{\bar{B}}^G (\delta_{\sm}^{-1}\delta_B)\big)$ is one-dimensional. To summarize, (\ref{emb}) is an embedding ${\bigotimes}_{v\in \Sigma_0}\pi_{x,v}^{K_v} \hookrightarrow \M_y'$. Finally, since $\hat{S}(K^p,E)_{\m}[\p_x]^{\an}$ clearly admits a $G$-invariant norm (the sup norm), Theorem 4.4.5 in \cite{Em06b} tells us that (\ref{emb}) is an isomorphism if $\delta$ is of non-critical slope. \end{proof}

To aid the reader we briefly recall the notion of non-critical slope: To each $\delta \in \hat{T}(E)$ we assign the element 
$\text{slp}(\delta) \in X^*(T \times_{\Q} E)$ defined as follows, cf. \cite[Def. 1.4.2]{Em06b}. First note that there is a natural surjection
$T(E)\twoheadrightarrow X_*(T \times_{\Q} E)$; the cocharacter $\mu_t \in X_*(T \times_{\Q} E)$ associated with $t\in T(E)$ is given by $\langle \chi,\mu_t\rangle=\text{ord}_E\chi(t)$ for all algebraic characters $\chi$ (here $\text{ord}_E$ is the valuation on $E$ normalized such that $\text{ord}_E(\varpi_E)=1$). Then the slope of $\delta$ is the algebraic character $\text{slp}(\delta)$ satisfying
$\langle \text{slp}(\delta),\mu_t \rangle=\text{ord}_E\delta(t)$ for all $t \in T$.

\begin{defn}
Let $\varrho=\frac{1}{2}\sum_{\alpha>0} \alpha$. We say that $\delta=\delta_{\alg}\delta_{\sm}$ is of non-critical slope if there is {\it{no}} simple root $\alpha$ for which the element
$s_{\alpha}(\delta_{\alg}+\varrho)+\text{slp}(\delta_{\sm})+\varrho$ lies in the $\Q_{\geq 0}$-cone generated by all simple roots.
\end{defn}

\section{Interpolation of the Weil-Deligne representations}\label{weildel}

Our goal in this section is to interpolate the Weil-Deligne representations $\WD(r_x|_{\Gal_{F_{\tilde{v}}}})$ across deformation space $X_{\bar{r}}
$, for a fixed $v \in \Sigma_0$. More precisely, for any affinoid subvariety $\Sp(A)\subset X_{\bar{r}}$ we will define a rank $n$ Weil-Deligne representation $\WD_{\bar{r},\tilde{v}}$
over $A$ such that
\begin{equation}\label{inter}
\WD(r_x|_{\Gal_{F_{\tilde{v}}}}) \simeq \WD_{\bar{r},\tilde{v}}\otimes_{A,x} \kappa(x)
\end{equation}
for all points $x \in \Sp(A)$. The usual proof of Grothendieck's monodromy theorem (cf. \cite[Cor. 4.2.2]{Tat79}) adapts easily to this setting, and this has already been observed by other authors. See for example \cite[7.8.3--7.8.14]{BC09}, \cite[5.2]{Pau11}, and \cite[4.1.6]{EH14}. 
To make our article more self-contained (and to point out the 'usual' assumption that $A$ is reduced is unnecessary) we give the details for the convenience of the reader. 

\begin{prop}\label{wd}
Let $w$ be a place of $F$ not dividing $p$, and let $A$ be an affinoid $E$-algebra. For any continuous representation
$\rho: \Gal_{F_w}\rightarrow \GL_n(A)$ there is a unique nilpotent $N \in M_n(A)$ such that the equality $\rho(\gamma)=\exp \big(t_p(\gamma)N\big)$ holds for all $\gamma$ in an open subgroup $J \subset I_{F_w}$. (Here $t_p: I_{F_w}\twoheadrightarrow \Z_p$ is a choice of homomorphism as in section (4.2) of \cite{Tat79}.)
\end{prop}

\begin{proof}
Choose a submultiplicative norm $\|\cdot\|$ on $A$ relative to which $A$ is complete (if $A$ is reduced one can take the spectral norm, cf. \cite[6.2.4]{BGR84}). Let $A^{\circ}$ be the (closed) unit ball. Then $I+p^i M_n(A^{\circ})$ is an open (normal) subgroup of $\GL_n(A^{\circ})$ for $i>0$, so its inverse image $\rho^{-1}(I+p^i M_n(A^{\circ}))=\Gal_{F_i}$ for some finite extension $F_i$ of $F_w$. Note that $F_{i+1}/F_i$ is a Galois extension whose Galois group is killed by $p$. Let us fix an $i>0$ and work with the restriction
$\rho|_{\Gal_{F_i}}$. Recall that wild inertia $P_{F_i}\subset I_{F_i}$ is the Sylow pro-$\ell$ subgroup where $w|\ell$. Since $\ell \neq p$ we deduce that $P_{F_i}\subset \Gal_{F_j}$ for all $j \geq i$. That is $\rho$ factors through the tame quotient $I_{F_i}/P_{F_i}\simeq \prod_{q\neq \ell} \Z_q$. For the same reason $\rho$ factors further through
$t_p: I_{F_i}\twoheadrightarrow \Z_p$. Therefore we find an element $\alpha \in I+p^i M_n(A^{\circ})$ (the image of $1\in \Z_p$ under $\rho$) such that $\rho(\gamma)=\alpha^{t_p(\gamma)}$ for all $\gamma \in I_{F_i}$. We let $N:=\log(\alpha)$. If we choose $i$ large enough ($i>1$ suffices, cf. the discussion in \cite[p. 220]{Sch11}) all power series converge and we arrive at $\rho(\gamma)=\exp\big(t_p(\gamma)N\big)$ for $\gamma \in I_{F_i}$. We conclude that we may take $J:=I_{F_2}$. (The uniqueness of $N$ follows by taking log on both sides.)

To see that $N$ is nilpotent note the standard relation $\rho(w)N\rho(w)^{-1}=\|w\|N$ for $w \in W_{F_i}$. If we take $w$ to be a (geometric) Frobenius this shows that all specializations of $N^n$ at points $x \in \Sp(A)$ are $0$ (by considering the eigenvalues in $\overline{\kappa(x)}$ as usual). Thus all matrix entries of $N^n$ are nilpotent (by the maximum modulus principle \cite[6.2.1]{BGR84}). Therefore $N$ itself is nilpotent since $A$ is Noetherian. \end{proof}

If we choose a geometric Frobenius $\Phi$ from $W_{F_w}$ (keeping the notation of the previous Proposition) we can thus define a Weil-Deligne representation
$(\tilde{\rho},N)$ on $A^n$ by the usual formula (\cite[4.2.1]{Tat79}):
$$
\rho(\Phi^s \gamma)=\tilde{\rho}(\Phi^s \gamma)\exp \big(t_p(\gamma)N\big)
$$
where $s \in \Z$ and $\gamma\in I_{F_w}$. With this definition $\tilde{\rho}: W_{F_w}\rightarrow \GL_n(A)$ is a representation which is trivial on the open subgroup $J\subset W_{F_w}$
(so continuous for the discrete topology on $A$). 

As already hinted at above we apply this construction to $r^{univ}|_{\Gal_{F_{\tilde{v}}}}$ for a fixed place $v \in \Sigma_0$, and an affinoid $\Sp(A)\subset X_{\bar{r}}$. We view the universal deformation $r^{univ}: \Gal_F \rightarrow \GL_n(R_{\bar{r}})$ as a representation on $A^n$ by composing with 
$R_{\bar{r}}\rightarrow \OO(X_{\bar{r}})\rightarrow A$. This gives a Weil-Deligne representation $\WD_{\bar{r},\tilde{v}}$
over $A$ with the interpolative property (\ref{inter}). 

\section{The local Langlands correspondence for $\GL_n$ after Scholze}\label{llc}

Scholze gave a new purely local characterization of the local Langlands correspondence in \cite{Sc13}. His trace identity (cf. Theorem 1.2 in loc. cit.) takes the following form. Let $\Pi$ be an irreducible smooth representation of $\GL_n(F_w)$, where $w$ is an arbitrary finite place of $F$. Suppose we are given $\tau=\Phi^s \gamma$ with $\gamma \in I_{F_w}$ and $s \in \Z_{>0}$, together with a $\Q$-valued 'cut-off' function $h \in \CC_c^{\infty}(\GL_n(\OO_{F_w}))$. First Scholze associates a $\Q$-valued function $\phi_{\tau,h}\in \CC_c^{\infty}(\GL_n(F_{w,s}))$, where $F_{w,s}$ denotes the unramified degree $s$ extension of $F_w$. The function $\phi_{\tau,h}$ is defined by taking the trace of $\tau\times h^{\vee}$ on (alternating sums of) certain formal nearby cycle sheaves \`{a} la Berkovich on deformation spaces of $\varpi$-divisible $\OO_{F_w}$-modules; and $h^{\vee}(g)=h({^tg}^{-1})$. See the discussion leading up to \cite[Thm. 2.6]{Sc13} for more details. Next one selects a function $f_{\tau,h}\in \CC_c^{\infty}(\GL_n(F_w))$ which is {\it{associated}} with $\phi_{\tau,h}$ 
in the sense that their (twisted) orbital integrals match. More precisely, with suitable normalizations one has the identity  $TO_{\delta}(\phi_{\tau,h})=O_{\gamma}(f_{\tau,h})$
for regular $\gamma=\mathcal{N}\delta$, cf. \cite[Thm. 2.1]{Clo87}. With our normalization of $\rec(\cdot)$ Scholze's trace identity reads
$$
\tr (f_{\tau,h}|\Pi)=\tr\big(\tau|\rec(\Pi \otimes |\det|^{(1-n)/2})\big) \cdot \tr(h|\Pi).
$$
We will make use of a variant of $f_{\tau,h}$ which lives in the Bernstein center of $\GL_n(F_w)$. We refer to section 3 of \cite{Hai11} for a succinct review of the basic properties and different characterizations of the Bernstein center. This variant $f_{\tau}$ has the property that $\tr(f_{\tau,h}|\Pi)=\tr(f_{\tau}*h|\Pi)$ and is defined for all $\tau \in W_{F_w}$ by decreeing that $f_{\tau}$ acts on any irreducible smooth representation $\Pi$ via scaling by 
$$
f_{\tau}(\Pi)=\tr\big(\tau|\rec(\Pi \otimes |\det|^{(1-n)/2}) \big).
$$
For the existence of $f_{\tau}$ see the proofs of \cite[Lem. 3.2]{Sc13}, \cite[Lem. 6.1]{Sc13a}, and/or \cite[Lem. 9.1]{Sc11}.

We apply this construction to each of the places $\tilde{v}$ with $v \in \Sigma_0$. Now $\tau=(\tau_{\tilde{v}})$ denotes a tuple of Weil elements
$\tau_{\tilde{v}} \in W_{F_{\tilde{v}}}$. Via our isomorphisms $i_{\tilde{v}}$ we view $f_{\tau_{\tilde{v}}}$ as an element of the Bernstein center of $U(F_v^+)$, and consider
$f_{\tau}:=\otimes_{v\in \Sigma_0} f_{\tau_{\tilde{v}}}$.

\begin{lem}\label{weil}
Let $x \in X_{\bar{r}}$ be arbitrary. Then $f_{\tau}$ acts on $\otimes_{v \in \Sigma_0} \pi_{x,v}$ via scaling by 
$$
f_{\tau}(\otimes_{v \in \Sigma_0} \pi_{x,v})=\prod_{v\in \Sigma_0} \tr \big( \tau_{\tilde{v}}| \WD(r_x|_{\Gal_{F_{\tilde{v}}}}) \big).
$$
\end{lem}

\begin{proof}
If $\{\pi_v\}_{v \in \Sigma_0}$ is a family of irreducible smooth representations, $f_{\tau}$ acts on $\otimes_{v \in \Sigma_0}\pi_v$ via scaling by 
$$
f_{\tau}(\otimes_{v \in \Sigma_0}\pi_v)=\prod_{v \in \Sigma_0} \tr \big(\tau_{\tilde{v}}|\rec(\BC_{\tilde{v}|v}(\pi_{v}) \otimes |\det|^{(1-n)/2})\big).
$$
Now use the defining property (\ref{pix}) of the representations $\pi_{x,v}$ attached to the point $x$.
\end{proof}

\section{Interpolation of traces}\label{tr}

Let $\frak{Z}(U(F_v^+))$ denote the Bernstein center of $U(F_v^+)$, and let $\mathcal{Z}(U(F_v^+),K_v)$ be the center of the Hecke algebra $\HH(U(F_v^+),K_v)$. There is a canonical homomorphism $\frak{Z}(U(F_v^+))\rightarrow \mathcal{Z}(U(F_v^+),K_v)$ obtained by letting the Bernstein center act on $\CC_c^{\infty}(K_v\backslash U(F_v^+))$, cf. \cite[3.2]{Hai11}.
We let $f_{\tau_{\tilde{v}}}^{K_v}$ be the image of $f_{\tau_{\tilde{v}}}$ under this map, and consider $f_{\tau}^{K_{\Sigma_0}}:=\otimes_{v \in \Sigma_0} f_{\tau_{\tilde{v}}}^{K_v}$
belonging to $\mathcal{Z}(K_{\Sigma_0}):=\otimes_{v \in \Sigma_0}\mathcal{Z}(U(F_v^+),K_v)$ which is the center of $\mathcal{H}(K_{\Sigma_0})$. In particular this operator $f_{\tau}^{K_{\Sigma_0}}$ acts on the sheaf $\M$ and its fibers $\M_y$. 

If $y=(x,\delta)\in Y(K^p,\bar{r})(E)$ is a classical point of non-critical slope, and we combine Proposition \ref{ncr} and Lemma \ref{weil}, we deduce that 
$f_{\tau}^{K_{\Sigma_0}}$ acts on $\M_y'\simeq \otimes_{v \in \Sigma_0} \pi_{x,v}^{K_v}$ via scaling by 
$$
\prod_{v\in \Sigma_0} \tr \big( \tau_{\tilde{v}}| \WD(r_x|_{\Gal_{F_{\tilde{v}}}}) \big).
$$
The goal of this section is to extrapolate this property to {\it{all}} points $y$. As a first observation we note that the above factor can be interpolated across deformation space $X_{\bar{r}}$. Indeed, let $\Sp(A)\subset X_{\bar{r}}$ be an affinoid subvariety and let $\WD_{\bar{r},\tilde{v}}$ be the Weil-Deligne representation on $A^n$ constructed after Proposition \ref{wd}. 

\begin{lem}\label{atau}
For each tuple $\tau=(\tau_{\tilde{v}}) \in \prod_{v\in \Sigma_0} W_{F_{\tilde{v}}}$ the element $a_{\tau}:=\prod_{v\in \Sigma_0} \tr \big( \tau_{\tilde{v}}| \WD_{\bar{r},\tilde{v}} \big) \in A$ 
satisfies the following interpolative property: For every point $x \in \Sp(A)$ the function $a_{\tau}$ specializes to
$$
a_{\tau}(x)=\prod_{v\in \Sigma_0} \tr \big( \tau_{\tilde{v}}| \WD(r_x|_{\Gal_{F_{\tilde{v}}}}) \big) \in \kappa(x).
$$ 
\end{lem}

\begin{proof}
This is clear from the interpolative property of $\WD_{\bar{r},\tilde{v}}$ by taking traces in (\ref{inter}).
\end{proof}

Our main result in this section (Proposition \ref{fam} below) shows that $a_{\tau}$ extends naturally to a function defined on the whole eigenvariety $Y(K^p,\bar{r})$ in such  a way that $f_{\tau}^{K_{\Sigma_0}}: \M \rightarrow \M$ is multiplication by $a_{\tau}$. 

%Now let $\Omega\subset \Sp(A)\times \hat{T}$ be an admissible open subset of $Y(K^p,\bar{r})$. Note that $\M$ lives on $X_{\bar{r}}\times \hat{T}$ but we 
%naturally view it as a coherent sheaf on its support $Y(K^p,\bar{r})$ (cf. \cite[Prop. 9.5.2/4]{BGR84}). The module of sections $\Gamma(\Omega,\M)$ carries a natural
%$\OO(\Omega)$-linear action of $\mathcal{H}(K_{\Sigma_0})$. (We denote the structure sheaf $\OO_{Y(K^p,\bar{r})}$ by $\OO$ to simplify the notation here.)
%We let $A$ act by composing with $A \rightarrow \OO(\Omega)$. 

%We now invoke the natural weight morphism $\omega: Y(K^p,\bar{r})\rightarrow \WW$ defined as the composition of maps
%$$
%Y(K^p,\bar{r}) \hookrightarrow X_{\bar{r}}\times \hat{T} \overset{\text{pr}}{\longrightarrow} \hat{T} \overset{\text{can}}{\longrightarrow} \WW.
%$$
%In general $Y(K^p,\bar{r})$ is admissibly covered by open affinoids $\Omega$ such that $\omega(\Omega)\subset \WW$ is an open affinoid, 
%$\omega: \Omega \rightarrow \omega(\Omega)$ is finite and surjective when restricted to any irreducible component of $\Omega$. 

%------------------INSERT J&N BELOW---------------

First we need to recall a couple of well-known facts from rigid analytic geometry. 

\begin{lem}\label{opensubisirred}
Let $X$ be an irreducible rigid analytic space (over some unspecified non-archimedean field) and let $Y \subset X$ be a non-empty Zariski open subset (cf. \cite[Def.~9.5.2/1]{BGR84}). Then $Y$ is irreducible.
\end{lem}

\begin{proof}
Let $\widetilde{X}\rightarrow X$ be the (irreducible) normalization of $X$. The pullback of $Y$ to $\widetilde{X}$ is a normalization $\widetilde{Y}\rightarrow Y$ and it suffices to show that the Zariski open subset $\widetilde{Y} \subset \widetilde{X}$ is connected (cf. \cite[Def.~2.2.2]{Con99}). Suppose $\widetilde{Y} = U\coprod V$ is an admissible covering with $U, V$ proper admissible open subsets of $\widetilde{Y}$. By Bartenwerfer's Hebbarkeitssatz \cite[p.~159]{Bar76} the idempotent function on $\widetilde{Y}$ which is $1$ on $U$ and $0$ on $V$ extends to an analytic function on $\widetilde{X}$, which is necessarily a non-trivial idempotent by the uniqueness in Bartenwerfer's Theorem "Riemann I". This contradicts the irreducibility of $\widetilde{X}$ (by \cite[Lem. ~2.2.3]{Con99}), so $\widetilde{Y}$ must be connected.
\end{proof} 

\begin{defn}
A Zariski dense subset $Z$ of a rigid space $X$ is called {\it{very Zariski dense}} (or {\it{Zariski dense and accumulation}}, see \cite[Prop.~2.6]{Che11}) if for $z \in Z$ and an affinoid open neighbourhood $z \in U \subset X$, there is an affinoid open neighbourhood $z \in V \subset U$ such that $Z\cap V$ is Zariski dense in $V$.
\end{defn}

\begin{lem}\label{opendense}
Let $X$ be a rigid space and let $Z \subset X$ be a very Zariski dense subset. Let $Y \subset X$ be a Zariski open subset which is Zariski dense. Then $Y\cap Z$ is very Zariski dense in $Y$.
\end{lem}	

\begin{proof}
We first note that it suffices to prove that $Y\cap Z$ is Zariski dense in $Y$. Very Zariski density then follows immediately from very Zariski density of $Z$ in $X$. We show that $Z$ is Zariski dense in every irreducible component of $Y$. By \cite[Cor.~2.2.9]{Con99} these irreducible components are given by the subsets $Y \cap C$ where $C$ is an irreducible 
component of $X$. Denote by $C^\circ$ the Zariski open subset of $X$ given by removing the intersections with all other irreducible components from 
$C$. Then $Y \cap C^\circ$ is irreducible by Lemma \ref{opensubisirred} and meets $Z$ since it is Zariski open in $X$. It follows from very Zariski density of $Z$ in $X$ that $Z$ is Zariski dense in $Y \cap C^\circ$. We deduce that $Z$ is Zariski dense in $Y \cap C$, as desired.
\end{proof}

%The key point is to show $Y\cap Z$ is Zariski {\it{dense}} in $Y$. The fact that it is indeed {\it{very}} Zariski dense is then trivial to check and left to the reader; it is a formal consequence of $Z$ being very Zariski dense. To show $Y\cap Z$ is dense, Lemma \ref{opensubisirred} takes care of the case where $X$ is irreducible. In the general case we proceed by showing
%$Y \cap Z$ has non-trivial intersection with every irreducible component of $Y$ as follows. 

%We let $C$ be an irreducible component of $X$. Since $Y\cap C$ is a non-empty Zariski open subset of $C$, it is irreducible by Lemma \ref{opensubisirred}, and we deduce from \cite[Cor.~2.2.9]{Con99} that $Y \cap C$ is an irreducible component of $Y$ -- and every irreducible component of $Y$ has this form. By (plain) Zariski density of $Z$ in $X$, 
%the intersection $Z \cap C \neq \varnothing$ is dense in $C$ and therefore intersects the open set $Y \cap C$. \end{proof}

%Let $U$ be an open affinoid neighbourhood of $z$ in $X$. By very Zariski density, we have an open affinoid neighbourhood $z \in V \subset U$ such that $Z\cap V$ is Zariski dense in $V$. Since $V\cap C \neq \varnothing$ (it contains $z$) is a union of irreducible components of $V$ (again by \cite[Cor.~2.2.9]{Con99}), $Z\cap V \cap C$ is Zariski dense in $V \cap C$. We deduce that $Z\cap C$ is Zariski dense in $C$, and therefore $Z \cap C \cap Y$ is non-empty. So $Z$ has non-trivial intersection with every irreducible component of $Y$. 

In order to deal with the {\it{non}-\'{e}tale} points below, the following generic freeness lemma will be crucial.

\begin{lem}\label{genfree}
Let $X$ be a reduced rigid space and let $\mathcal{M}$ be a coherent $\OO_X$-module. Then there is a Zariski open and dense subset $X_{\mathcal{M}} \subset X$ over which $\mathcal{M}$ is locally free.
\end{lem}

\begin{proof}
We follow an argument from the proof of \cite[Thm.~5.1.2]{Han17}: The regular locus $X^{reg}$ of $X$ is Zariski open and dense, by the excellence of affinoid algebras. If $U \subset X$ is 
an affinoid open $\mathcal{M}$ is locally free at a regular point $x \in U$ if and only if $x$ is not in the support of $\oplus_{i = 1}^{\dim U}\mathrm{Ext}^i_{\OO(U)}(\mathcal{M}(U),\OO(U))$. This shows that $\mathcal{M}$ is locally free over a Zariski open subset $X_\mathcal{M}$ which is the intersection of $X^{reg}$ and another Zariski open subset of $X$ -- the complement of the support. Namely, if $U \subset X^{reg}$ is a connected affinoid open (so $\OO(U)$ is a regular domain) then the support of $\oplus_{i = 1}^{\dim U}\mathrm{Ext}^i_{\OO(U)}(\mathcal{M}(U),\OO(U))$ in $\Spec(\OO(U))$ has dimension $< \dim(U)$, by \cite[Cor.~3.5.11(c)]{BrH93} and therefore its complement is dense. We deduce that $X_\mathcal{M}$ is dense in $X$. 
\end{proof}

The following observation lies at the heart of our interpolation argument.

\begin{lem}\label{density}
Let $w:X \rightarrow W$ be a map of reduced equidimensional rigid spaces and let $\mathcal{M}$ be a coherent $\OO_X$-module. We assume that $X$ admits a  covering by affinoid opens $V$ such that 

\begin{itemize}
\item[(1)] $w(V) \subset W$ is affinoid open,
\item[(2)] The restriction $w|_V: V \rightarrow w(V)$ is finite,
\item[(3)] $\mathcal{M}(V)$ is a finite projective $\OO(w(V))$-module.
\end{itemize}
Let $Z\subset X$ be a very Zariski dense subset, and suppose $\phi \in \End_{\OO_X}(\mathcal{M})$ induces the zero map $\phi_z=0$ on the fibers $\mathcal{M}_z=\mathcal{M}\otimes_{\OO_X}\kappa(z)$ for all $z \in Z$. Then $\phi=0$. 
\end{lem}

\begin{proof}
First we restrict to the Zariski open and dense set $X_\mathcal{M}$ from Lemma \ref{genfree}. Since $\mathcal{M}$ is locally free over $X_\mathcal{M}$, the locus in $X_\mathcal{M}$ where $\phi$ vanishes is a Zariski closed subset. By Lemma \ref{opendense}, this locus also contains a Zariski dense set of points (namely $Z \cap X_{\mathcal{M}}$) so we infer that $\phi|_{X_\mathcal{M}}=0$.
    
Now we let $V \subset X$ be an affinoid open forming part of the cover described in the statement. Let $w(V)_0\subset w(V)$ be the (Zariski open and dense -- since $W$ is reduced) locus where the map $V \rightarrow w(V)$ is finite \'{e}tale. 
    
Since $X \backslash X_\mathcal{M} \subset X$ is a Zariski closed subset of dimension $<\dim X$, the set $W_1 := w\big(V \cap (X \backslash X_\mathcal{M})\big)$ is a Zariski closed subset of $w(V)$ with dimension $<\dim X=\dim W$. So $w(V) \backslash W_1$ is Zariski open and dense in $w(V)$.
    
We deduce that $w(V)_0 \cap (w(V) \backslash W_1)$ is a Zariski dense subset of $w(V)$. Moreover, $\phi$ induces the zero map on the fibers $\mathcal{M}(V)\otimes_{\OO(w(V))}\kappa(y)$ for all $y$ in this dense intersection: Use that $w|_V$ is \'etale at $y$, so if $x_{1},\ldots,x_{r}$ are the preimages of $y$ in $V$, then 
$$
\mathcal{M}(V)\otimes_{\OO(w(V))}\kappa(y) \simeq \bigoplus_{i=1}^{r}\mathcal{M}(V) \otimes_{\OO(V)}\kappa(x_{i}) 
$$
and we know that $\phi$ acts as zero on each $\mathcal{M}(V) \otimes_{\OO(V)}\kappa(x_{i})$ since $x_i \in X_{\mathcal{M}}$ (otherwise $y=w(x_i)\in W_1$), as observed in the first paragraph of the proof. We conclude that $\phi=0$ on $\mathcal{M}(V)$: Indeed $\mathcal{M}(V)$ is a finite projective $\OO(w(V))$-module so the points $y \in w(V)$ where 
$\phi$ vanishes on the fiber form a Zariski closed subset which contains $w(V)_0 \cap (w(V) \backslash W_1)$. Since $W$ is reduced $\phi_{\mathcal{M}(V)}=0$. Since $V$ was arbitrary, we must have $\phi=0$ on $\mathcal{M}$ as desired.
\end{proof}

We now return to the notation of section \ref{ev}. We have defined the eigenvariety $Y(K^p, \overline{r})$ to be the (scheme-theoretic) support of the coherent sheaf 
$\mathcal{M}$ over $X_{\overline{r}}\times \hat{T}$. It comes equipped with a natural weight morphism $\omega: Y(K^p,\bar{r})\rightarrow \WW$ defined as the composition of maps
$$
Y(K^p,\bar{r}) \hookrightarrow X_{\bar{r}}\times \hat{T} \overset{\text{pr}}{\longrightarrow} \hat{T} \overset{\text{can}}{\longrightarrow} \WW.
$$
The following Proposition summarises some important facts about $Y(K^p, \overline{r})$ and $\omega$.

\begin{lem}\label{ingredients}

The eigenvariety $Y(K^p, \overline{r})$ satisfies the following properties.

\begin{itemize}
\item[(1)] $Y(K^p, \overline{r})$ has an admissible cover by open affinoids $(U_i)_{i \in I}$ such that for all $i$ there exists an open affinoid $W_i\subset \mathcal{W}$ which fulfills 
(a) and (b) below;

\begin{itemize}
\item[(a)] The weight morphism $\omega: Y(K^p, \overline{r})\rightarrow \mathcal{W}$ induces, upon restriction to each irreducible component $C \subset U_i$, a finite surjective map $C \rightarrow W_i$.
\item[(b)] Each $\OO(U_i)$ is isomorphic to an $\OO(W_i)$-subalgebra of $\End_{\OO(W_i)}(P_i)$ for some finite projective $\OO(W_i)$-module $P_i$. 
\end{itemize}

\item[(2)] The classical points of non-critical slope are very Zariski dense in $Y(K^p, \overline{r})$.
\item[(3)] $Y(K^p, \overline{r})$ is reduced.
\end{itemize}
\end{lem} 

\begin{proof}
These can be proved in a similar way to the analogous statements in \cite{BHS16}. More precisely, we refer to Prop.~3.11, Thm.~3.19 and Cor.~3.20 of that paper. (Note that in the proof of Cor.~3.20 we can, in our setting, replace the reference to \cite{CEG$^+$16} with the well-known assertion that the Hecke operators at good places act semisimply on spaces of cuspidal 
automorphic forms.)
\end{proof}

Since $Y(K^p,\bar{r})$ projects to $X_{\bar{r}}$, its ring of functions $\OO(Y(K^p, \overline{r}))$ becomes an $R_{\bar{r}}$-algebra via the natural map
$R_{\bar{r}}\rightarrow \OO^0(X_{\bar{r}})$. Pushing forward the universal deformation of $\overline{r}$ (with a fixed choice of basis) then yields a continuous representation 
$$
r: \Gal_F \rightarrow \GL_n\big(\OO(Y(K^p, \overline{r}))\big).
$$
In particular, for every open affinoid $U \subset Y(K^p, \overline{r})$ we may specialize $r$ further and arrive at a continuous representation 
$r: \Gal_F \rightarrow \GL_n\big(\OO(U)\big)$. We may in fact take $\OO^0(U)$ here (the functions bounded by one), but we will not need that.

It follows from Proposition \ref{wd} that for $v\in \Sigma_0$, an open affinoid $U \subset Y(K^p, \overline{r})$, and a fixed choice of lift of geometric Frobenius $\Phi=\Phi_{\tilde{v}}$ in $W_{F_{\tilde{v}}}$, we obtain a Weil--Deligne representation $\mathrm{WD}_{\overline{r},\tilde{v}}(U)$ over $\OO(U)$. Moreover, this construction is obviously compatible as we vary $U$ in the sense that if $U' \subset U$, then $\mathrm{WD}_{\overline{r},\tilde{v}}(U)$ pulls back to $\mathrm{WD}_{\overline{r},\tilde{v}}(U')$ over $U'$ (by the uniqueness in Proposition \ref{wd}). To be precise, there is a natural isomorphism of Weil--Deligne representations over $\OO(U')$,
$$
\mathrm{WD}_{\overline{r},\tilde{v}}(U') \simeq \mathrm{WD}_{\overline{r},\tilde{v}}(U) \otimes_{\OO(U)}\OO(U').
$$
Now, for a tuple of Weil elements $\tau = (\tau_{\tilde{v}}) \in \prod_{v \in \Sigma_0} W_{F_{\tilde{v}}}$ we obtain functions 
$$
a_{\tau,U}:=\prod_{v\in\Sigma_0} \tr\big((\tau_{\tilde{v}}|\mathrm{WD}_{\overline{r},\tilde{v}}(U)\big) \in \OO(U)
$$
as defined above in Lemma \ref{atau}. By the compatibility just mentioned, $a_{\tau,U'}=\text{res}_{U,U'}(a_{\tau,U})$ when $U'\subset U$. It follows that we may glue the 
$a_{\tau,U}$ and get a function $a_{\tau}=a_{\tau,Y(K^p, \overline{r})}$ on the whole eigenvariety $Y(K^p,\bar{r})$ with the interpolation property in Lemma \ref{atau}. 

\begin{prop}\label{fam}
The operator $f_\tau^{K_{\Sigma_0}}$ acts on $\mathcal{M}$ via scaling by $a_\tau$, for every $\tau \in \prod_{v\in\Sigma_0}W_{F_{\tilde{v}}}$.
\end{prop}

\begin{proof}
We must show the endomorphism $\phi:=f_\tau^{K_{\Sigma_0}}-a_\tau$ of $\mathcal{M}$ equals zero. By the discussion at the beginning of this section (just prior to \ref{atau}) we know $\phi$ induces the zero map on the fibres of $\mathcal{M}$ at classical points of non-critical slope. We are now done by Lemma \ref{density} 
(together with Lemma \ref{ingredients}).
\end{proof}

By specialization at any point $y=(x,\delta)\in Y(K^p,\bar{r})$ we immediately find that $f_{\tau}^{K_{\Sigma_0}}$ acts on the fiber $\M_y$ (and hence its dual $\M_y'$) via scaling by $a_{\tau}(x)$. We summarize this below.

\begin{cor}\label{et}
Let $y \in Y(K^p,\bar{r})$ be an arbitrary point. Then $f_{\tau}^{K_{\Sigma_0}}$ acts on $\M_y'$ via scaling by 
$$
\prod_{v\in \Sigma_0} \tr \big( \tau_{\tilde{v}}| \WD(r_x|_{\Gal_{F_{\tilde{v}}}}) \big).
$$
\end{cor}

\begin{proof}
This is an immediate consequence of Proposition \ref{fam}.
%Pick a small enough neighborhood $\Omega$ around $y$ such that $\omega: \Omega \overset{\sim}{\longrightarrow} \omega(\Omega)$ and use the previous Proposition \ref{fam}. 
\end{proof}

%-----------------------------------------NEW SECTION--------

\section{Interpolation of central characters}\label{central}

In this section we will reuse parts of the argument from the previous section \ref{tr} to interpolate the central characters $\omega_{\pi_{x,v}}$ across the eigenvariety. We include it here mostly for future reference. It will only be used in this paper in the very last paragraph of Remark \ref{ext} below.

For $v \in \Sigma_0$ we let $Z(U(F_v^+))$ be the center of $U(F_v^+)$ (recall that its {\it{Bernstein}} center is denoted by $\frak{Z}$). There is a natural homomorphism
$Z(U(F_v^+))\rightarrow \mathcal{Z}(U(F_v^+),K_v)^{\times}$ which takes $\xi_v$ to the double coset operator $[K_v \xi_vK_v]$. Taking the product over $v \in \Sigma_0$ we get an analogous map $Z(U(F_{\Sigma_0}^+))\rightarrow \mathcal{Z}(K_{\Sigma_0})^{\times}$ which we will denote $\xi=(\xi_v)_{v\in \Sigma_0}\mapsto h_{\xi}^{K_{\Sigma_0}}=
\otimes_{v\in \Sigma_0} [K_v \xi_vK_v]$. Thus $h_{\xi}^{K_{\Sigma_0}}$ operates on $\M$ and its fibers. 

If $y=(x,\delta)\in Y(K^p,\bar{r})(E)$ is a classical point of non-critical slope the action of $h_{\xi}^{K_{\Sigma_0}}$ on $\M_y'\simeq \otimes_{v\in \Sigma_0}\pi_{x,v}^{K_v}$ is clearly just multiplication by $\prod_{v\in \Sigma_0}\omega_{\pi_{x,v}}(\xi_v)$. This property extrapolates to {\it{all}} points $y$ by mimicking the proof in section \ref{tr}, as we will now explain.

For $\Sp(A) \subset X_{\bar{r}}$ we have the Weil-Deligne representation $\WD_{\bar{r},\tilde{v}}$ on $A^n$. Consider its determinant 
$\det(\WD_{\bar{r},\tilde{v}})$ as a character $F_{\tilde{v}}^{\times}\rightarrow A^{\times}$ via local class field theory. Note that $Z(U(F_v^+))\simeq Z(\GL_n(F_{\tilde{v}}))\simeq F_{\tilde{v}}^{\times}$ which allows us to view the product $\prod_{v\in \Sigma_0} \det(\WD_{\bar{r},\tilde{v}})$ as a character $\omega:Z(U(F_{\Sigma_0}^+))\rightarrow A^{\times}$. 
Clearly the specialization of $\omega$ at any $x \in \Sp(A)$ is $\omega_x=\otimes_{v\in \Sigma_0}\omega_{\pi_{x,v}}:  Z(U(F_{\Sigma_0}^+))\rightarrow \kappa(x)^{\times}$ by the interpolative property of 
$\WD_{\bar{r},\tilde{v}}$.

By copying the proof of Proposition \ref{fam} almost verbatim, one easily deduces the following.

\begin{prop}\label{center}
There is a homomorphism $\omega: Z(U(F_{\Sigma_0}^+))\rightarrow \OO(Y(K^p,\bar{r}))^{\times}$ such that $h_{\xi}^{K_{\Sigma_0}}: \M \rightarrow \M$ is 
multiplication by $\omega(\xi)$ for all $\xi$. In particular, for any point $y=(x,\delta)\in Y(K^p,\bar{r})$, the action of $h_{\xi}^{K_{\Sigma_0}}$ on $\M_y'$ is scaling by 
$\prod_{v\in \Sigma_0}\omega_{\pi_{x,v}}(\xi_v)$.
\end{prop}

%------------------------------------------------------------------------

\section{Proof of the main result}\label{pf}

We now vary $K_{\Sigma_0}$ and reinstate the notation $\M_{K^p}$ (instead of just writing $\M$) to stress the dependence on $K^p=K_{\Sigma_0}K^{\Sigma}$. Suppose $K_{\Sigma_0}'\subset K_{\Sigma_0}$
is a compact open subgroup, and let $K'^p=K_{\Sigma_0}'K^{\Sigma}$. Recall that the global sections of $\M_{K^p}$ is the {\it{dual}} of $J_B(\hat{S}(K^p,E)_{\m}^{\an})$. Thus we find a natural transition map $\M_{K'^p}\twoheadrightarrow \M_{K^p}$ of sheaves on $X_{\bar{r}}\times \hat{T}$. Taking their support we find that $Y(K^p,\bar{r})\hookrightarrow Y(K'^p,\bar{r})$. Passing to the dual fibers at a point $y\in Y(K^p,\bar{r})$ yields an embedding $\M_{K^p,y}' \hookrightarrow \M_{K'^p,y}'$ which is equivariant for the Hecke action
(i.e., compatible with the map $\mathcal{H}(K_{\Sigma_0}')\twoheadrightarrow \mathcal{H}(K_{\Sigma_0})$ given by $e_{K_{\Sigma_0}}\star(\cdot)\star e_{K_{\Sigma_0}}$).
The limit $\varinjlim_{K_{\Sigma_0}} \M_{K^p,y}'$ thus becomes an admissible representation of $U(F_{\Sigma_0}^+) \overset{\sim}{\longrightarrow} \prod_{v \in \Sigma_0 }\GL_n(F_{\tilde{v}})$ with coefficients in $\kappa(y)$. Subsequently we will use the next lemma to show it is of finite length.

\begin{lem}\label{sq}
Let $y \in Y(K^p,\bar{r})$ be any point. Let $\otimes_{v \in \Sigma_0} \pi_v$ be an arbitrary irreducible subquotient\footnote{Such exist by Zorn's lemma; any finitely generated subrepresentation admits an irreducible quotient.} of 
$\varinjlim_{K_{\Sigma_0}} \M_{K^p,y}'$. Then for all places $v \in \Sigma_0$ we have an isomorphism
$$
\WD(r_x|_{\Gal_{F_{\tilde{v}}}})^{ss}\simeq \rec(\BC_{\tilde{v}|v}(\pi_{v}) \otimes |\det|^{(1-n)/2})^{ss}.
$$
(Here $ss$ means semisimplification of the underlying representation $\tilde{\rho}$ of $W_{F_{\tilde{v}}}$, and setting $N=0$.) 
\end{lem}

\begin{proof}
By Lemma \ref{et} we know that $f_{\tau}$ acts on $\varinjlim_{K_{\Sigma_0}} \M_{K^p,y}'$ via scaling by $a_{\tau}(x)$. On the other hand, by the proof of Lemma \ref{weil} we know 
what $f_{\tau}(\otimes_{v \in \Sigma_0}\pi_v)$ is. By comparing the two expressions we find that
$$
\prod_{v\in \Sigma_0} \tr \big( \tau_{\tilde{v}}| \WD(r_x|_{\Gal_{F_{\tilde{v}}}}) \big)=\prod_{v \in \Sigma_0} \tr \big(\tau_{\tilde{v}}|\rec(\BC_{\tilde{v}|v}(\pi_{v}) \otimes |\det|^{(1-n)/2})\big)
$$
for all tuples $\tau$. This shows that $\WD(r_x|_{\Gal_{F_{\tilde{v}}}})$ and $\rec(\BC_{\tilde{v}|v}(\pi_{v}) \otimes |\det|^{(1-n)/2})$ have the same semisimplification 
for all $v \in \Sigma_0$ by 'linear independence of characters'. \end{proof}

We employ Lemma \ref{sq} to show $\varinjlim_{K_{\Sigma_0}} \M_{K^p,y}'$ has finite length (which for an admissible representation is equivalent to being finitely generated by Howe's Theorem, cf. \cite[4.1]{BZ76}). 

\begin{lem}\label{length}
The length of $\varinjlim_{K_{\Sigma_0}} \M_{K^p,y}'$ as a $U(F_{\Sigma_0}^+)$-representation is finite, and uniformly bounded in $y$ on quasi-compact subvarieties of 
$Y(K^p,\bar{r})$.
\end{lem}

\begin{proof}
We first show finiteness. Suppose the direct limit is of infinite length, and choose an infinite proper chain of $U(F_{\Sigma_0}^+)$-invariant subspaces
$$
\varinjlim_{K_{\Sigma_0}} \M_{K^p,y}' =V_0 \supset V_1 \supset V_2 \supset V_3 \supset \cdots \y \y \y \y V_i \neq V_{i+1}.
$$
Taking $K_{\Sigma_0}$-invariants (which is exact as $\text{char}_E=0$) we find a decreasing chain of $\mathcal{H}(K_{\Sigma_0})$-submodules $V_i^{K_{\Sigma_0}}\subset \M_{K^p,y}'$. The fiber is finite-dimensional so this chain must become stationary. I.e., $V_i/V_{i+1}$ has {\it{no}} nonzero $K_{\Sigma_0}$-invariants for $i$ large enough. If we can show that every irreducible subquotient $\otimes_{v \in \Sigma_0} \pi_v$ of $\varinjlim_{K_{\Sigma_0}} \M_{K^p,y}'$ {\it{has}} nonzero $K_{\Sigma_0}$-invariants, we are done. We will show that we can find a small enough $K_{\Sigma_0}$ with this last property.

The local Langlands correspondence preserves $\epsilon$-factors, and hence conductors. (See \cite{JPSS} for the definition of conductors in the $\GL_n$-case, and \cite[p. 21]{Tat79} for the Artin conductor of a Weil-Deligne representation.) Therefore, for every place $v \in \Sigma_0$ we get a bound on the conductor of $\BC_{\tilde{v}|v}(\pi_{v})$:
\begin{equation}\label{bd}
\begin{split}
c(\pi_v) &:= c(\BC_{\tilde{v}|v}(\pi_{v})) \\
   &=c\big( \rec(\BC_{\tilde{v}|v}(\pi_{v}) \otimes |\det|^{(1-n)/2})\big)\\
   &\leq c\big( \rec(\BC_{\tilde{v}|v}(\pi_{v}) \otimes |\det|^{(1-n)/2})^{ss}\big)+n \\
   &\overset{\ref{sq}}{=}c\big(\WD(r_x|_{\Gal_{F_{\tilde{v}}}})^{ss}\big)+n.
\end{split}
\end{equation}
(In the inequality we used the following general observation: If $(\tilde{\rho},\mathcal{N})$ is a Weil-Deligne representation on a vector space $S$, its conductor is
$c(\tilde{\rho})+\dim S^I - \dim (\ker \mathcal{N})^I$, where $I$ is shorthand for inertia; $c(\tilde{\rho})$ is the usual Artin conductor, which is clearly invariant under semisimplification: $c(\tilde{\rho})$ only depends on $\tilde{\rho}|_I$ which is semisimple because it has finite image.)
This shows $c(\pi_v)$ is bounded in terms of $x$. If we take $K_{\Sigma_0}$ small enough, say $K_{\Sigma_0}=\prod_{v \in \Sigma_0} K_v$ where
$$
K_v=i_{\tilde{v}}^{-1}\{g \in \GL_n(\OO_{F_{\tilde{v}}}): (g_{n1},\ldots,g_{nn})\equiv (0,\ldots,1) \text{ mod $\varpi_{F_{\tilde{v}}}^N$} \}
$$
with $N$ greater than the right-hand side of the inequality (\ref{bd}), then every constituent $\otimes_{v \in \Sigma_0} \pi_v$ as above satisfies $\pi_v^{K_v}\neq 0$ as desired. This shows the length is finite.

To get a uniform bound in $K^p$ and $\bar{r}$ we improve on the bound (\ref{bd}) using \cite[Prop. ~1.1]{Liv89}: Since $r_x|_{\Gal_{F_{\tilde{v}}}}$ is a lift of 
$\bar{r}|_{\Gal_{F_{\tilde{v}}}}$ the aforequoted Proposition implies that 
$$
c\big(\WD(r_x|_{\Gal_{F_{\tilde{v}}}}) \big)\leq c(\bar{r}|_{\Gal_{F_{\tilde{v}}}})+n.
$$
(One can improve this bound but the point here is to get uniformity.) Taking $K_{\Sigma_0}$ as above with $N$ greater than $c(\bar{r}|_{\Gal_{F_{\tilde{v}}}})+2n$ the above argument guarantees that 
the $U(F_{\Sigma_0}^+)$-length of $\varinjlim_{K_{\Sigma_0}} \M_{K^p,y}'$ is the same as the $\mathcal{H}(K_{\Sigma_0})$-length of $\M_{K_{\Sigma_0}K^{\Sigma},y}'$ 
which is certainly at most $\dim_E \M_{K_{\Sigma_0}K^{\Sigma},y}'$. This dimension is uniformly bounded when $y$ is constrained to a quasi-compact subspace of 
$Y(K^p,\bar{r})$.
\end{proof}

\subsection{Strongly generic representations}\label{generic}

Recall the definition of $\pi_{x,v}$ in (\ref{pix}). We call $x$ a {\it{generic}} point if $\pi_{x,v}$ is a generic representation (i.e., when it has a Whittaker model). For instance, all classical points are generic (cf. the proof of Lemma \ref{unique}). We will impose a stronger condition on $r_x|_{\Gal_{F_{\tilde{v}}}}$ which ensures that $\pi_{x,v}$ is fully induced from a supercuspidal representation of a Levi subgroup (thus in particular is generic, cf. \cite{BZ77}). This rules out that $\pi_{x,v}$ is Steinberg for instance, and bypasses difficulties arising from having nonzero monodromy.

\begin{defn}\label{gen}
Decompose $\WD(r_x|_{\Gal_{F_{\tilde{v}}}})^{ss}\simeq \tilde{\rho}_1\oplus \cdots \oplus \tilde{\rho}_t$ into a sum of irreducible representations
$\tilde{\rho}_i:W_{F_{\tilde{v}}}\rightarrow \GL_{n_i}(\bar{\Q}_p)$. We say $r_x|_{\Gal_{F_{\tilde{v}}}}$ is {\it{strongly generic}} if $\tilde{\rho}_i \nsim \tilde{\rho}_j\otimes \epsilon$ for all $i \neq j$,
where $\epsilon: \Gal_{F_{\tilde{v}}}\rightarrow \Z_p^{\times}$ is the cyclotomic character.
\end{defn}

For the rest of this section we will assume $r_x$ is strongly generic at each $v \in \Sigma_0$. In the notation of Definition \ref{gen}, each $\tilde{\rho}_i$ corresponds to a supercuspidal representation $\tilde{\pi}_i$ of $\GL_{n_i}(F_{\tilde{v}})$ and 
$$
\pi_{x,v}\simeq \Ind_{P_{n_1,\ldots,n_t}}^{\GL_n}(\tilde{\pi}_1 \otimes \cdots \otimes \tilde{\pi}_t)
$$
since the induced representation is irreducible, cf. \cite{BZ77}. Indeed $\tilde{\pi}_i \nsim \tilde{\pi}_j(1)$ for all $i \neq j$. (The twiddles above $\rho_i$ and $\pi_i$ should not be confused with taking the contragredient.)

By Lemma \ref{sq} the factor $\pi_v$ of any irreducible subquotient $\otimes_{v \in \Sigma_0} \pi_v$ of $\varinjlim_{K_{\Sigma_0}} \M_{K^p,y}'$ has the same supercuspidal support as $\pi_{x,v}$. Since the latter is fully induced from $P_{n_1,\ldots,n_t}$ they must be isomorphic.
In summary we have arrived at the result below.

\begin{cor}
Let $y=(x,\delta) \in Y(K^p,\bar{r})$ be a point at which $r_x$ is strongly generic at every $v \in \Sigma_0$. Then $\varinjlim_{K_{\Sigma_0}} \M_{K^p,y}'$ has finite length, and every irreducible subquotient is isomorphic to $\otimes_{v \in \Sigma_0} \pi_{x,v}$. 
\end{cor}

Altogether this proves Theorem \ref{main} in the Introduction. 

\begin{rem}\label{ext}
Naively one might hope to remove the '$ss$' in Theorem \ref{main} by showing that $\pi_{x,v}$ has no non-split self-extensions; $\Ext_{\GL_n(F_{\tilde{v}})}^1(\pi_{x,v},\pi_{x,v})=0$.
However, this is false even if we assume $\pi_{x,v}\simeq \Ind_P^{\GL_n}(\sigma)$ with $\sigma=\otimes_{j=1}^t \tilde{\pi}_j$ supercuspidal (as above). Let us explain why. For 
simplicity we assume $\sigma$ is regular, which means $w\sigma\simeq \sigma \Rightarrow w=1$ for all block-permutations $w\in S_n$. In other words $\tilde{\pi}_i\not \simeq\tilde{\pi}_j$ for $i \neq j$ with $n_i=n_j$. Under this assumption the 'geometric lemma' (cf. \cite[Prop. 6.4.1]{Cas95}) gives an actual direct sum decomposition of the $N$-coinvariants:
$$
(\pi_{x,v})_N\simeq \oplus_{w} w\sigma
$$
with $w$ running over block-permutations as above. The usual adjointness property of $(-)_N$ is easily checked to hold for $\Ext^i$ (cf. \cite[Prop. 2.9]{Pra13}). Therefore 
$$
\Ext_{\GL_n}^1(\pi_{x,v},\pi_{x,v})\simeq \Ext_M^1((\pi_{x,v})_N, \sigma)\simeq \prod_w\Ext_M^1(w\sigma, \sigma)\simeq \Ext_M^1(\sigma,\sigma).
$$
In the last step we used \cite[Cor. 5.4.4]{Cas95} to conclude that $\Ext_M^1(w\sigma, \sigma)=0$ for $w\neq 1$. However, $\Ext_M^1(\sigma,\sigma)$ is always non-trivial. For example, consider the principal series case where $P=B$ and $\sigma$ is a smooth character of $T$. Here $\Ext_T^1(\sigma,\sigma)\simeq 
\Ext_T^1({\bf{1}},{\bf{1}})\simeq \Hom(T,E)\simeq E^n$. In general, if $\sigma$ is an irreducible representation of $M$ with central character $\omega$, there is a short exact sequence
$$
0 \longrightarrow \Ext_{M,\omega}^1(\sigma,\sigma)\longrightarrow \Ext_{M}^1(\sigma,\sigma) \longrightarrow \Hom(Z_M,E) \longrightarrow 0
$$
(cf. \cite[Prop. 8.1]{Pas10} whose proof works verbatim with coefficients $E$ instead of $\bar{\F}_p$). If $\sigma$ is supercuspidal it is projective and/or injective in the category of smooth $M$-representations with central character $\omega$, and vice versa (cf. \cite[Thm. 5.4.1]{Cas95} and \cite{AR04}). In particular $\dim_E\Ext_{M}^1(\sigma,\sigma)=\dim(Z_M)$.

By Proposition \ref{center} all the self-extensions of $\pi_{x,v}$ arising from $\varinjlim_{K_{\Sigma_0}} \M_{K^p,y}'$ actually live in the full subcategory of smooth representations with central character $\omega_{\pi_{x,v}}$. As we just pointed out, supercuspidal is equivalent to being projective and/or injective in this category. Thus at least in the case where $\otimes_{v\in \Sigma_0}\pi_{x,v}$ is supercuspidal we can remove the '$ss$' in Theorem \ref{main}.
\end{rem}

%------------------------------------Newton m_y--------------------

\begin{rem}\label{my}
We comment on the multiplicity $m_y$ in the analogous case of $\GL(2)_{/\Q}$. Replacing our unitary group $U$ with $\GL(2)_{/\Q}$, and replacing $\hat{S}(K^p,E)$
with the completed cohomology of modular curves $\hat{H}^1(K^p)_E$ with tame level $K^p \subset \GL_2(\A_f^p)$, a statement analogous to Theorem \ref{main} is a consequence of
Emerton's local--global compatibility theorem \cite[Thm.~1.2.1]{Em11b}, under the assumption that $\overline{r}|_{\Gal_{\Q_p}}$ is not isomorphic to a twist of $\begin{psmallmatrix}
1 & *\\0 & 1\end{psmallmatrix}$ or $\begin{psmallmatrix} 1 & *\\0 & \overline{\epsilon}\end{psmallmatrix}$. With this assumption, the multiplicities $m_y$ are (at least predicted to be) equal to
$2$ (coming from the two-dimensional Galois representation $r_x$), and the representations of $\GL_2(\Q_{\Sigma_0})$ which appear are semisimple.

Indeed, it follows from \emph{loc.~cit.} that we have $m_y = 2\dim_E J_B^\delta(\Pi(\varrho_x)^\mathrm{an})$ where $\varrho_x:=r_x|_{\Gal_{\Q_p}}$. When $\varrho_x$ is absolutely irreducible, it follows from \cite[Thm.~1.1, Thm.~1.2]{Dos14} (see also \cite[Thm.~0.6]{Col14}) that $J_B^\delta(\Pi(\varrho_x)^\mathrm{an})$ has dimension at most $1$. If $\varrho_x$ is reducible, then  \cite[Conj.~3.3.1(8), Lem.~4.1.4]{Em06d} predicts that $J_B^\delta(\Pi(\varrho_x)^\mathrm{an})$ again has dimension at most $1$, unless $\varrho_x$ is of the form $\eta \oplus \eta$ for some continuous character $\eta:\Gal_{\Q_p}\rightarrow E^{\times}$.

In the exceptional case with $\varrho_x\simeq \eta \oplus \eta$ scalar, where \cite[Thm.~1.2.1 (2)]{Em11b} does not apply, we have $\dim_E J_B^{\delta}(\Pi(\varrho_x)^\mathrm{an}) =2$ when $\delta = \eta|\cdot|\otimes\eta\epsilon|\cdot|^{-1}$ and therefore \cite[Conj.~1.1.1]{Em11b} predicts that we have $m_y = 4$ for $y=(x,\eta|\cdot|\otimes\eta\epsilon|\cdot|^{-1})$. Again the representation of $\GL_2(\Q_{\Sigma_0})$ which appears is predicted to be semisimple.
\end{rem}

%--------------------------------------------------------------------------

\subsection{The general case at Iwahori level}\label{iwahori}

In this section we assume $\bar{r}$ is automorphic of tame level $K^p=K_{\Sigma_0}K^{\Sigma}$ where $K_{\Sigma_0}=\prod_{v \in \Sigma_0} K_v$ is a product of {\it{Iwahori}} subgroups. This can usually be achieved by a solvable base change; i.e. by replacing $\bar{r}$ with its restriction $\bar{r}|_{\Gal_{F'}}$ for some solvable Galois extension $F'/F$
(cf. the 'Skinner-Wiles trick' \cite{SW01}). We make this assumption to employ a genericity criterion of Barbasch-Moy \cite{BM94}, which was recently strengthened by Chan-Savin
in \cite{CSa17a} and \cite{CSa17b}.

\subsubsection{Genericity and Iwahori-invariants}

The setup of \cite{CSa17a} is the following. Let $G$ be a split group over a $p$-adic field $F$, with a choice of Borel subgroup $B=TU$. We assume these are defined over $\OO=\OO_F$, and let $I \subset G(\OO)$ be the Iwahori subgroup (the inverse image of $B$ over the residue field $\F_q$). The Iwahori-Hecke algebra $\mathcal{H}$ has basis 
$T_w=[IwI]$ where $w \in W_{\text{ex}}$ runs over the extended affine Weyl group $W_{\text{ex}}=N_G(T)/T(\OO)$. The basis vectors satisfy the usual relations
\begin{itemize}
\item $T_{w_1}T_{w_2}=T_{w_1w_2}$ when $\ell(w_1w_2)=\ell(w_1)+\ell(w_2)$,
\item $(T_s-q)(T_s+1)=0$ when $\ell(s)=1$.
\end{itemize}
Here $\ell: W_{\text{ex}}\rightarrow \Z$ denotes the length function defined by $q^{\ell(w)}=|IwI/I|$. Inside of $\mathcal{H}$ we have the subalgebra $\mathcal{H}_W$ of functions supported on $G(\OO)$, which has basis $\{T_w\}_{w \in W}$ where $W$ is the (actual) Weyl group. The algebra $\mathcal{H}_W$ carries a natural one-dimensional representation
$\text{sgn}: \mathcal{H}_W \rightarrow \C$ which sends $T_w \mapsto (-1)^{\ell(w)}$, and we are interested in the sgn-isotypic subspaces of $\mathcal{H}$-modules.

\begin{defn}\label{spart}
For a smooth $G$-representation $\pi$ (over $\C$) we introduce the following subspace of the Iwahori-invariants
$$
\Bbb{S}(\pi)=\bigcap_{w\in W}\big(\pi^I\big)^{T_w=(-1)^{\ell(w)}}.
$$
In other words the (possibly trivial) subspace of $\pi^I$ where $\mathcal{H}_W$ acts via the $\text{sgn}$-character. 
\end{defn}

Fix a non-trivial continuous unitary character $\psi: F \rightarrow \C^{\times}$ and extend it to a character of $U$ as in \cite[Sect. 4]{CSa17a}. For a smooth $G$-representation $\pi$ 
we let $\pi_{U,\psi}$ be the 'top derivative' of $\psi$-coinvariants (whose dual is exactly the space of $\psi$-Whittaker functionals on $\pi$). 

\begin{thm}\label{savin} (Barbasch-Moy, Chan-Savin).
Let $\pi$ be a smooth $G$-representation which is generated by $\pi^I$. Then the natural map $\Bbb{S}(\pi) \hookrightarrow \pi \twoheadrightarrow \pi_{U,\psi}$ is an isomorphism. 
\end{thm}

\begin{proof}
This is \cite[Cor. 4.5]{CSa17a} which is a special case of \cite[Thm. 3.5]{CSa17b}.
\end{proof}

In particular, an irreducible representation $\pi$ with $\pi^I\neq 0$ is {\it{generic}} if and only if $\Bbb{S}(\pi)\neq 0$, in which case $\dim \Bbb{S}(\pi)=1$. This is the genericity criterion we will use below. 

\subsubsection{The $\Bbb{S}$-part of the eigenvariety}

We continue with the usual setup and notation. We run the eigenvariety construction with $\hat{S}(K^p,E)_{\m}$ replaced by its $\Bbb{S}$-subspace. More precisely, for each $v \in \Sigma_0$ we have the functor $\Bbb{S}_v$ (Def. \ref{spart}) taking smooth $\GL_n(F_{\tilde{v}})$-representations to vector spaces over $E$. We apply their composition $\Bbb{S}=\circ_{v\in \Sigma_0} \Bbb{S}_v$ to $\varinjlim_{K_{\Sigma_0}} \hat{S}(K^p,E)_{\m}$. I.e., we take
$$
\Pi:=\bigcap_{v \in \Sigma_0}\bigcap_{w\in W_v}\big(\hat{S}(K^p,E)_{\m}\big)^{T_w=(-1)^{\ell(w)}}.
$$
Clearly $\Pi$ is a closed subspace of $\hat{S}(K^p,E)_{\m}$, and therefore an admissible Banach representation of $G=G(\Q_p)$. As a result $J_B(\Pi^{\text{an}})'$ is coadmissible 
(cf. \cite[Prop. 3.4]{BHS16}) and hence the global sections $\Gamma(X_{\bar{r}}\times \hat{T},\M_{\Pi})$ of a coherent sheaf $\M_{\Pi}$ on $X_{\bar{r}}\times \hat{T}$. We let
$$
Y_{\Pi}(K^p,\bar{r})=\text{supp}(\M_{\Pi})
$$
be its schematic support with the usual annihilator ideal sheaf. Mimicking the proof of Lemma \ref{fib} we obtain the following description of the dual fiber of $\M_{\Pi}$ at a point $y=(x,\delta)\in Y_{\Pi}(K^p,\bar{r})$;
$$
\M_{\Pi,y}'\simeq J_B^{\delta}(\Pi[\p_x]^{\an})\simeq \bigcap_{v \in \Sigma_0}\bigcap_{w\in W_v} J_B^{\delta}(\hat{S}(K^p,E)_{\m}[\p_x]^{\an})^{T_w=(-1)^{\ell(w)}}.
$$
This clearly shows $Y_{\Pi}(K^p,\bar{r})$ is a closed subvariety of $Y(K^p,\bar{r})$. Our immediate goal is to show equality.

\begin{lem}\label{eq}
$Y_{\Pi}(K^p,\bar{r})=Y(K^p,\bar{r})$.
\end{lem} 

\begin{proof}
Since the classical points are Zariski dense in $Y(K^p,\bar{r})$ we just have to show each classical $y=(x,\delta)$ in fact lies in $Y_{\Pi}(K^p,\bar{r})$. Let $\pi$ be an automorphic representation such that $r_x \simeq r_{\pi,\iota}$. This is an irreducible Galois representation (since $\bar{r}$ is) and thus $\text{BC}_{F/F^+}(\pi)$ is a cuspidal and therefore generic automorphic representation of $\GL_n(\A_F)$. In particular the factors of $\otimes_{v\in \Sigma_0} \pi_v$ are generic. Taking $T_w$-eigenspaces of the embedding 
$\otimes_{v\in \Sigma_0} \pi_v^{K_v}\hookrightarrow \M_y'$ from Prop. \ref{ncr} yields a map $\otimes_{v\in \Sigma_0} \Bbb{S}_v(\pi_v) \hookrightarrow \M_{\Pi,y}'$. Finally, by 
Theorem \ref{savin} we conclude that $\otimes_{v\in \Sigma_0} \Bbb{S}_v(\pi_v)\neq 0$ so that $\M_{\Pi,y}'\neq 0$.
\end{proof}

\subsubsection{Conclusion}

Now let $y \in Y(K^p,\bar{r})$ be an arbitrary point. By Lemma \ref{eq} we now know $\M_{\Pi,y}'\neq 0$. Note that $\M_{\Pi,y}'=\Bbb{S}(\varinjlim_{K_{\Sigma_0}}\M_y')$ and we immediately infer that $\varinjlim_{K_{\Sigma_0}}\M_y'$ does have {\it{some}} generic constituent (by \ref{savin}). 

Suppose $\otimes_{v\in \Sigma_0}\pi_v$ is {\it{any}} generic constituent of $\varinjlim_{K_{\Sigma_0}}\M_y'$. Lemma \ref{sq} tells us $\pi_v$ and $\pi_{x,v}$ have the same supercuspidal support. By the theory of Bernstein-Zelevinsky derivatives $\Ind_{P_{n_1,\ldots,n_t}}^{\GL_n}(\tilde{\pi}_1 \otimes \cdots \otimes \tilde{\pi}_t)$ has a {\it{unique}}
generic constituent (where the $\tilde{\pi}_i$ are supercuspidals as before). Consequently, there is a unique generic representation $\pi_{x,v}^{\text{gen}}$ with the same supercuspidal support as $\pi_{x,v}$, and $\pi_v \simeq  \pi_{x,v}^{\text{gen}}$.

We summarize our findings below.

\begin{thm}\label{iwcase}
Let $y=(x,\delta)\in Y(K^p,\bar{r})$ be an arbitrary point, where $K_{\Sigma_0}$ is a product of Iwahori subgroups. Then the following holds:
\begin{itemize}
\item[(1)] $\otimes_{v\in \Sigma_0} \pi_{x,v}^{\text{gen}}$ \underline{occurs} as a constituent of 
$\varinjlim_{K_{\Sigma_0}}\M_y'$ (possibly with multiplicity).
\item[(2)] Every generic constituent of $\varinjlim_{K_{\Sigma_0}}\M_y'$ is isomorphic to $\otimes_{v\in \Sigma_0} \pi_{x,v}^{\text{gen}}$.
\end{itemize}
Here $\pi_{x,v}^{\text{gen}}$ is the generic representation of $\GL_n(F_{\tilde{v}})$ with the same supercuspidal support as $\pi_{x,v}$. 
\end{thm}

It would be interesting to relax the assumption that $K_v$ is Iwahori for $v \in \Sigma_0$. In \cite{CSa17b} they consider more general $\frak{s}$ in the Bernstein spectrum of 
$\GL_{mr}$ (where the Levi is $\GL_r \times \cdots \times \GL_r$ and the supercuspidal representation is $\tau \otimes \cdots \otimes \tau$). For such an $\frak{s}$-type $(J,\rho)$ one can identify the Hecke algebra $\mathcal{H}(J,\rho)$ with the Iwahori-Hecke algebra of $\GL_m$ -- but over a possibly larger $p$-adic field. This is used to define the subalgebra 
$\mathcal{H}_{S_m}\subset \mathcal{H}(J,\rho)$ which carries the sgn-character. If $\pi \in \mathcal{R}^{\frak{s}}(\GL_{mr})$ is an admissible representation,
\cite[Thm. 3.5]{CSa17b} shows that a certain adjunction map $\Bbb{S}_{\rho}(\pi) \rightarrow \pi_{U,\psi}$ is an isomorphism, where $\Bbb{S}_{\rho}(\pi)$ denotes the sgn-isotypic subspace of $\Hom_J(\rho,\pi)$. (In the case $r=1$ and $\tau={\bf{1}}$ this recovers Theorem \ref{savin} above; the type is $(I,{\bf{1}})$.) Instead of considering 
$\hat{S}(K_{\Sigma_0}K^{\Sigma},E)_{\m}$ in the eigenvariety construction one could take $K_{\Sigma_0}=\prod_{v \in \Sigma_0} J_v$ and $\rho=\otimes_{v\in \Sigma_0}\rho_v$ for certain types $(J_v,\rho_v)$ and consider the space $\Hom_{K_{\Sigma_0}}(\rho,\hat{S}(K^{\Sigma},E)_{\m})$ which would result in an eigenvariety $Y_{\rho}(K_{\Sigma_0}K^{\Sigma},\bar{r})$ which of course sits as a closed subvariety of $Y(K_{\Sigma_0}'K^{\Sigma},\bar{r})$ for $K_{\Sigma_0}'\subset \ker(\rho)$. If we take an arbitrary point 
$y \in Y_{\rho}(K_{\Sigma_0}K^{\Sigma},\bar{r})$ we know $\varinjlim_{K_{\Sigma_0}}\M_y'$ lies in the $\frak{s}_v$-component (for each $v \in \Sigma_0$) and it is at least plausible the above arguments with $\Bbb{S}$ replaced by $\Bbb{S}_{\rho}$ would allow us to draw the same conclusion: $\varinjlim_{K_{\Sigma_0}}\M_y'$ admits $\otimes_{v\in \Sigma_0} \pi_{x,v}^{\text{gen}}$ as its unique generic irreducible subquotient (up to multiplicity). The inertial classes $\frak{s}$ considered in \cite{CSa17b} are somewhat limited. However, Savin has communicated to us a more general (unpublished) genericity criterion -- without restrictions on $\frak{s}$.

%-------------------------------------

%\appendix

%\section{Cont a}

%\section{Diff b}

%--------------------------------------

%\subsection*{Acknowledgments} We would like to thank

%--------------------------------------

%-----------------------------------------------------------------------------------------------------------

\bigskip

\noindent {\it{E-mail address}}: {\texttt{hchristian.johansson@googlemail.com}}

\noindent {\sc{DPMMS, Centre for Math. Sciences, Wilberforce Road, Cambridge CB3 0WB, UK.}}

\medskip

\noindent {\it{E-mail address}}: {\texttt{j.newton@kcl.ac.uk}}

\noindent {\sc{Department of Mathematics, King's College London, Strand, London, WC2R 2LS, UK.}}

\medskip

\noindent {\it{E-mail address}}: {\texttt{csorensen@ucsd.edu}}

\noindent {\sc{Department of Mathematics, UCSD, La Jolla, CA, USA.}}

\end{document}